\documentclass[11pt]{amsart}
\usepackage{amscd, amsmath, amssymb, amsthm, amsfonts, amstext, verbatim, enumitem, graphicx, mathtools, xfrac, tikz-cd, microtype, nameref, thmtools}
\usepackage[capitalize]{cleveref}

\definecolor{darkblue}{rgb}{0.0, 0.0, 0.6}

\hypersetup{
  pdfauthor={Sergey Sinchuk, Andrei Smolensky},
  pdftitle={Decompositions of congruence subgroups of Chevalley groups},  
  colorlinks=true,
  urlcolor=darkblue,
  linkcolor=darkblue,
  citecolor=darkblue}

\usepackage[hyperref=true, backend=biber, citestyle=numeric-comp, sortlocale=en_US, url=false, doi=false, eprint=true, maxbibnames=4]{biblatex}            
\renewbibmacro*{volume+number+eid}{\ifentrytype{article}{\- \iffieldundef{volume}{}{Vol.~\printfield{volume},}\iffieldundef{number}{}{ No.~\printfield{number},}}}
\renewbibmacro{in:}{\ifentrytype{article}{}{\printtext{\bibstring{in}\intitlepunct}}}
\newbibmacro{string+doi}[1]{\iffieldundef{doi}{\iffieldundef{url}{#1}{\href{\thefield{url}}{#1}}}{\href{https://doi.org/\thefield{doi}}{#1}}}
\DeclareFieldFormat[article, inproceedings, inbook, thesis]{title}{\usebibmacro{string+doi}{\mkbibquote{#1}}}

\usepackage[margin=2cm]{geometry}

\newlist{thmlist}{enumerate}{1} \setlist[thmlist]{label={\rm(\arabic{thmlisti})}, ref=\thethm.(\arabic{thmlisti}),noitemsep} \Crefname{thmlisti}{Theorem}{Theorems}
\newlist{lemlist}{enumerate}{1} \setlist[lemlist]{label={\rm(\arabic{lemlisti})}, ref=\thelemma.(\arabic{lemlisti}),noitemsep} \Crefname{lemlisti}{Lemma}{Lemmas}

\theoremstyle{plain}

\newtheorem{thm}{Theorem}
\Crefname{thm}{Theorem}{Theorems}
\numberwithin{equation}{section}

\newtheorem{lemma}{Lemma}
\numberwithin{lemma}{section}
\Crefname{lemma}{Lemma}{Lemmas}

\newtheorem{cor}[lemma]{Corollary}
\Crefname{cor}{Corollary}{Corollaries}

\newtheorem{prop}[lemma]{Proposition}
\Crefname{prop}{Proposition}{Propositions}

\newtheorem*{thm*}{Theorem}
\newtheorem*{lemma*}{Lemma}

\theoremstyle{definition}

\newtheorem{dfn}[lemma]{Definition}
\Crefname{dfn}{Definition}{Definitions}

\Crefname{example}{Example}{Examples}

\theoremstyle{remark}

\Crefname{rem}{Remark}{Remarks}

\DeclareMathOperator{\K}{K}

\DeclareMathOperator{\G}{G}
\DeclareMathOperator{\GL}{GL}
\DeclareMathOperator{\SL}{SL}

\DeclareMathOperator{\E}{E}
\DeclareMathOperator{\EP}{EP}

\DeclareMathOperator{\Hh}{H}
\DeclareMathOperator{\U}{U}

\DeclareMathOperator{\M}{M}
\DeclareMathOperator{\SR}{SR}
\DeclareMathOperator{\sr}{sr}
\DeclareMathOperator{\asr}{asr}
\DeclareMathOperator{\shape}{shape}

\DeclareMathOperator{\Max}{Max}
\DeclareMathOperator{\Spec}{Spec}

\DeclareMathOperator{\Epin}{Epin}
\DeclareMathOperator{\Stab}{Stab}
\DeclareMathOperator{\ASR}{ASR}
\DeclareMathOperator{\Ums}{Ums}
\DeclareMathOperator{\Umd}{Umd}
\DeclareMathOperator{\rk}{rk}
\newcommand{\rA}{\mathsf{A}}
\newcommand{\rB}{\mathsf{B}}
\newcommand{\rC}{\mathsf{C}} 
\newcommand{\rD}{\mathsf{D}} 
\newcommand{\rE}{\mathsf{E}}
\newcommand{\rF}{\mathsf{F}}
\newcommand{\rG}{\mathsf{G}}

\makeatletter
\newcommand{\indexbox}[1]{\text{\fboxsep=.1em\fbox{\m@th$\displaystyle#1$}}}
\makeatother
   
\title[Decompositions of congruence subgroups of Chevalley groups]{Decompositions of congruence subgroups \\ of Chevalley groups}
\author{Sergey Sinchuk}
\address{Chebyshev laboratory, St. Petersburg State University, St. Petersburg, Russia}
\email{sinchukss {\it at} yandex.ru;}

\author{Andrei Smolensky}
\address{Department of Mathematics and Mechanics, St. Petersburg State University, St. Petersburg, Russia}
\email{andrei.smolensky {\it at} gmail.com}
\thanks{The authors of the present paper acknowledge the financial support of the Russian Science Foundation grant 14-11-00297}

\date {\today}

\begin{document}

\begin{abstract} 
We formulate and prove relative versions of several decompositions known in the theory of Chevalley groups over commutative rings.
These decompositions are used to obtain factorizations in terms of subsystem subgroups of type $\rA_\ell$ and 
 upper estimates of the width of principal congruence subgroups with respect to Tits--Vaserstein generators.
%Some of our results are new even in the absolute case and were previously studied only for groups over finite fields.
\end{abstract}

\maketitle

\section{Introduction}\label{sec:intro}
This paper studies decompositions of Chevalley groups over commutative rings (see e.\,g.~\cite{VP} for the introduction to Chevalley groups over rings).
Recall that classical decompositions of Chevalley groups over fields such as Bruhat decomposition do not generalize to groups over more general rings.
Nevertheless, Chevalley groups over rings of {\it finite stable rank} do admit several remarkable ``parabolic decompositions'', i.\,e. decompositions formulated in terms of parabolic and unipotent subgroups.
The interest in these decompositions comes, in particular, from the study of stability problems for $\K_1$-functors modeled on Chevalley groups (see e.\,g.~\cite{St78, Si13}).

The first goal of this paper is to obtain analogues of these parabolic decompositions for congruence subgroups of Chevalley groups.
One of our main results is the following theorem, which is, essentially, a version of Dennis--Vaserstein decomposition for the relative elementary group $\E(\Phi, R, I)$
 (cf.~\cite[Theorems~2.5 and~4.1]{St78},~\cite[Theorem~1.2]{Si13}).
\begin{thm}\label{thm:DennisVaserstein}
Let $\Phi \neq \rE_8$ be an irreducible root system of rank $\ell\geqslant 2$ and $\{ r, s \}$ be a pair of distinct endnodes of the Dynkin diagram of $\Phi$.
Denote by $d$ the distance between $r$ and $s$ on the Dynkin diagram.
Assume that $\Phi$, $I$ and $\{r, s\}$ satisfy either of the following assumptions:
\begin{thmlist}
 \item $\sr(I) \leqslant d$ for $\Phi$ classical or $\Phi=\rG_2, \rF_4$;
 \item $\sr(I) \leqslant d$ for $\Phi=\rE_6,\rE_7$ with $\{r, s\} = \{2, \ell\}$;
 \item $\asr(I) \leqslant d$ for $\Phi=\rE_6,\rE_7$ with $\{r, s\} = \{1, \ell \}$, 
\end{thmlist}
Then $\E(\Phi, R, I) = \EP_r(R, I) \cdot \U(\Sigma^-_r \cap \Sigma^-_s, I) \cdot \EP_s(R, I).$
\end{thm}
The notation for the groups and ring-theoretic invariants appearing in the statement of the above theorem is introduced in sections~\ref{sec:preliminaries} and~\ref{sec:stability-conditions}.
Other relative decompositions studied in this article are Gauss and Bass--Kolster decompositions (see~\cref{thm:srRI1,thm:BassKolster}).

As an application of parabolic factorizations, we deduce several {\it subgroup factorizations} of finite width for relative Chevalley groups.
For example, as a corollary of Bass--Kolster decomposition we obtain upper estimates of the number of factors in the presentation of a classical Chevalley group
 as a product of its subsystem subgroups $\SL(2, R, I)$, see~\cref{thm:SL2width}.
As an application of~\cref{thm:DennisVaserstein} we prove the following theorem inspired by the main result of~\cite{Nik07}.
\begin{thm}\label{thm:spin-sln-prod} Assume that $\sr(I) \leqslant 2$. Then the group $\Epin(2\ell, R, I)=\E(\rD_\ell, R, I)$ is a product of at most $9$ regularly embedded subgroups of type $\rA_{\ell-1}$. \end{thm}
As another application, we obtain results on the bounded generation of the relative elementary group.
Recall from~\cite[Theorem~2]{Va86} that $\E(\Phi, R, I)$ can be generated by the set of Stein--Tits--Vaserstein generators
 $z_\alpha(s, \xi) =  x_{-\alpha}(-\xi) x_\alpha(s) x_{-\alpha}(\xi)$, where $s\in I$ and $\xi\in R$
(in fact, a smaller generating set suffices, see~\cref{prop:Stepanov-theorem}).
In section~\ref{sec:boundgen} we deduce from an earlier result of O.~Tavgen that the width $\E(\Phi, R, I)$ with respect to this generating set is
 finite for Dedekind rings of arithmetic type having a real embedding.
Moreover, as a further application of parabolic factorizations, we obtain in~\cref{thm:width} concrete estimates of such width
 for ideals of the ring $\mathbb{Z}[\sfrac{1}{p}]$ and also for ideals of stable rank $1$.

\section{Preliminaries}\label{sec:preliminaries}
Let $\Phi \subseteq \mathbb{R}^\ell$ be a reduced irreducible root system of rank $\ell$ with a fixed basis of simple roots $\Pi=\{\alpha_1, \ldots, \alpha_l\}$.
We use the conventional numbering of basis vectors of $\Pi$ which follows Bourbaki (see~\cite[Table~1]{PSV98}).
For a root $\alpha\in\Phi$ we denote by $m_i(\alpha)$ the $i$-th coefficient in the expansion of $\alpha$ in $\Pi$,
 i.\,e. $\alpha = \sum_{i=1}^n m_i(\alpha) \alpha_i$.

A proper closed subset of roots $S\subseteq \Phi$ is called {\it parabolic} (resp. {\it reductive}, resp. {\it special}) if $\Phi=S \cup -S$ (resp. $S = -S$, resp. $S \cap -S=\varnothing$).
Any parabolic subset $S \subseteq \Phi$ can be decomposed into the disjoint union of its \emph{reductive} and \emph{special} parts, i.\,e. 
$S = \Sigma_S \sqcup \Delta_S$, where $\Sigma_S \cap (-\Sigma_S) = \varnothing$, $\Delta_S = -\Delta_S$.

We denote by $(\alpha, \beta)$ the scalar product of roots and by $\langle \alpha, \beta\rangle$ the integer $2(\alpha, \beta)/(\beta, \beta)$.
Recall that fundamental weights $\varpi_1,\ldots,\varpi_\ell$ are the vectors of $\mathbb{R}^\ell$ that are characterized by the property that $\langle \varpi_i, \alpha_j \rangle = \delta_{ij}$.
 
Denote by $W(\Phi)$ the Weyl group, i.\,e. subgroup of isometries of $\Phi$ generated by all simple reflections $\sigma_\alpha$, $\alpha\in\Phi$.
For a reductive subset $\Delta\subseteq \Phi$ denote by $W(\Delta)$ the subgroup of $W(\Phi)$ generated by $\sigma_\alpha$ for $\alpha\in\Delta$.

For a nonempty $J\subseteq \Pi$ consider the following subsets of roots:
\begin{align*}
& \Delta_J = \{\alpha \in \Phi \mid \shape(J, \alpha)=0\}, \\
& \Sigma^\pm_J = \{\alpha \in \Phi \mid \shape(J, \alpha) \in \mathbb{Z}_{>0} \Phi^\pm \}, \\
& S_J^\pm = \Delta_J \sqcup \Sigma_J^\pm,
\end{align*}
where the $J$-shape of a root $\beta\in \Phi$ is defined by the formula: $\shape(J, \beta)=\sum\nolimits_{i\in J} m_i(\beta) \alpha_i.$
Clearly, $\Delta_J$ is a reductive subset, while $S^\pm_J$ and $\Sigma^\pm_J$ are parabolic and special subsets, respectively.
For two disjoint subsets $I, J\subseteq \Pi$ one has 
\[ \Sigma^\pm_{I \cup J} = \Sigma^\pm_I\cup\Sigma^\pm_J, \quad \Delta_{I\cup J} = \Delta_I \cap \Delta_J. \]
We omit curly braces in the above notations when $J$ is a one- or two-element set, e.\,g. $\Delta_k=\Delta_{\{k\}}$ and $\Sigma_{i, j}^\pm=\Sigma_{\{i, j\}}^\pm$, etc.

\begin{lemma}[{\cite[Lemma~1]{ABS}}]\label{lemma:abs}
Let $\alpha, \beta \in \Sigma^\pm_J$ be a pair of roots of the same length such that $\shape(J, \alpha)=\shape(J, \beta)\neq 0$.
Then $\alpha$ and $\beta$ are conjugate under the action of $W(\Delta_J)$.
\end{lemma}

Let $\G(\Phi, R)$ be the simply-connected Chevalley group of type $\Phi$ over an arbitrary commutative ring $R$ and let $\E(\Phi, R)$ be its elementary subgroup,
i.\,e. the subgroup generated by the elementary root unipotents $x_\alpha(\xi)$, $\alpha\in\Phi$, $\xi\in R$, see~\cite{VP, St78, S}.

For an ideal $I \trianglelefteq R$ we denote by $\G(\Phi, R, I)$ the \emph{principal congruence subgroup} of level $I$.
Recall that the \emph{relative elementary Chevalley subgroup} $\E(\Phi, R, I) \leqslant \G(\Phi, R, I)$ is defined as the normal closure in $\E(\Phi, R)$ of the subgroup 
$\E(\Phi, I)$ generated by the set $\mathcal{X} = \{x_\alpha(s) \mid \alpha\in \Phi, \ s \in I \}$.

For a subset of roots $S\subseteq\Phi$ denote by $\mathcal{Z}(S)$ the union of $\mathcal{X}$ and $\{ z_\alpha(s, \xi) \mid s\in I, \ \xi \in R, \ \alpha \in S\}.$
\begin{prop}[{\cite[Theorem~3.4]{S}}] \label{prop:Stepanov-theorem} Let $\Phi$ be an irreducible root system of rank $\geq 2$ and $S \subseteq \Phi$ be an arbitrary parabolic subset of roots with special part $\Sigma_S$.
 Then the relative elementary subgroup $\E(\Phi, R, I)$ is generated as an abstract group by $\mathcal{Z}(\Sigma_S)$. \end{prop}
For an ideal $I$ denote by $I^{\indexbox{2}}$ the ideal of $R$ generated by the squares of elements of $I$.
 \begin{lemma}[{\cite[Corollary~3.3]{S}}]\label{lemma:Stepanov-ideal}
Let $\Phi$ be a root system of rank $\geqslant2$, let $R$ be a commutative ring and $I\trianglelefteq R$ be its ideal.
If $\Phi\neq\rC_\ell$ then $\E\left(\Phi, R, I^2\right)\leqslant\E(\Phi, I)$, otherwise $\E\left(\Phi, R, II^{\indexbox{2}}\right)\leqslant\E(\Phi, I)$.
\end{lemma}
For $\varepsilon\in R^*$ set $w_\alpha(\varepsilon) = x_\alpha(\varepsilon) x_{-\alpha}(-\varepsilon^{-1}) x_{\alpha}(\varepsilon).$
If $\rk(\Phi)\geqslant 2$ the following relation holds:
\begin{equation}\label{rel:R3}
w_\alpha(\varepsilon) x_{\beta}(\xi) w_\alpha(\varepsilon)^{-1} =
x_{\sigma_{\alpha}\beta} \left(\eta_{\alpha, \beta}\cdot\varepsilon^{-\langle\beta, \alpha \rangle}\xi\right), \quad \varepsilon\in R^*, \ \xi\in R.
\end{equation}
where $\eta_{\alpha, \beta}=\pm 1$. The coefficients $\eta_{\alpha, \beta}$ can be expressed in terms of the structure constants of the corresponding Lie algebra (see~\cite[\S13]{VP}).
For a reductive subset $\Delta \subseteq \Phi$ denote by $\widetilde{W}(\Delta)$ the \emph{extended Weyl group}, i.\,e. the subgroup of $\E(\Phi, R)$ generated by all $w_{\alpha}(1)$, $\alpha \in \Delta$.
\begin{lemma} \label{lemma:weylfacts} Let $\Phi$ be an irreducible root system and let $I$ be an ideal of $R$. 
\begin{lemlist}
\item \label{item-trans1} For every two roots $\alpha, \beta \in \Phi$ of the same length there exists $w \in \widetilde{W}(\Phi)$ such that $X_{\alpha}(I)^w = X_\beta(I)$.
\item \label{item-trans2} Let $\alpha_s\in \Pi$ be a fundamental root and $\alpha, \beta \in \Sigma^\pm_s$ be a pair of roots of the same length such that $m_s(\alpha) = m_s(\beta)$.
 Then there exists $w\in \widetilde{W}(\Delta_s)$ such that $X_\alpha(I)^w = X_\beta(I)$. \end{lemlist}
\end{lemma}
\begin{proof}
The first assertion follows from \eqref{rel:R3} and the fact that $W(\Phi)$ acts transitively on the set of roots of the same length.
The second assertion follows from \eqref{rel:R3} and \cref{lemma:abs}.
\end{proof}

Recall that semisimple root elements $h_\alpha(\varepsilon)$, $\alpha\in\Phi$, $\varepsilon\in R^*$ are defined as $h_\alpha(\varepsilon)=w_\alpha(\varepsilon)w_\alpha(-1)$.
These elements satisfy the following relation:
\begin{alignat}{2} 
& \phantom{[} h_\alpha(\varepsilon)x_\beta(\xi)h_\alpha(\varepsilon)^{-1} = x_\beta\left(\varepsilon^{\langle\beta, \alpha\rangle}\xi\right), && \alpha, \beta \in \Phi, \ \varepsilon \in R^*, \xi \in R. \label{rel:h-w}
\end{alignat}

For a special subset of roots $\Sigma\subseteq \Phi$ we denote by $\U(\Sigma, I)$ the subgroup spanned by root subgroups $X_{\alpha}(I)$ for $\alpha\in \Sigma$. 
For $J\subset\Pi$ the subgroup $\U(\Sigma_J^+, I)$ is normalized by $\E(\Delta_J, R)$, hence the Minkowski product set $\EP_J(R, I) = \E(\Delta_J, R, I) \cdot \U(\Sigma_J^+, I)$ is a subgroup, which we call a \emph{standard elementary parabolic subgroup}.
In the sequel the following two identities will be referred to as \emph{Levi decomposition}:
\begin{equation} \label{rel:Levi-decomp} \EP_J(R, I) = \U(\Sigma_J^+, I) \cdot \E(\Delta_J, R, I) = \E(\Delta_J, R, I) \cdot \U(\Sigma_J^+, I). \end{equation}
When $J = \{ \alpha_s \}$ for some $1 \leq s\leq \ell$, we write $\EP_s(R, I)$ as a shorthand for $\EP_{\{s\}}(R, I)$.
We also use the notation $\EP_J(R)$ instead of $\EP_J(R, R)$.

Denote by $\Hh(\Phi, R)$ the subgroup generated by all $h_\alpha(\varepsilon)$, $\alpha\in\Phi$, $\varepsilon\in R^*$, and set
\[ \Hh(\Phi, R, I) = \Hh(\Phi, R)\cap\G(\Phi, R, I)=\langle h_\alpha(\varepsilon), \ \alpha\in\Phi, \ \varepsilon\in R^*\cap(1+I)\rangle. \]
It is not hard to see that $\Hh(\Phi, R, I) \leqslant \E(\Phi, R, I)$.
Indeed, if $s\in I$ is such that $1+s\in R^*$, then $(1+s)^{-1}-1 \in I$ and we can factor $h_\alpha(\varepsilon)$ into a product of elements of $\E(\Phi, R, I)$ as follows:
\begin{multline} \label{eq:rel-tor-elementary}
h_\alpha(1+s) = x_\alpha\left(-1\middle)\, x_{-\alpha}\middle(-s\middle)\, x_\alpha\middle((1+s)^{-1}\middle)\, x_{-\alpha}\middle(s(1+s)\right) = \\
= x_\alpha\left((1+s)^{-1}-1\middle)\, z_{-\alpha}\middle(-s, (1+s)^{-1}\middle)\, x_{-\alpha}\middle(s(1+s)\right).
\end{multline}

In the sequel we use elementary facts about representations of Chevalley groups (see e.\,g.~\cite{PSV98,VP} for a more detailed introduction).
Recall that, by definition, the {\it cone of dominant weights} $P_{++}(\Phi)$ consists of all nonnegative integral linear combinations of fundamental weights $\varpi_1,\ldots \varpi_\ell$.
To each element $\mu \in P_{++}(\Phi)$ one can associate a representation $\pi$ of $\G(\Phi, R)$ on a free $R$-module $V = V(\mu)$, called {\it Weyl module}.
For our purposes it suffices to restrict attention to representations with fundamental highest weight (i.\,e. $\mu=\varpi_i$ for some $i$), which are, moreover, {\it basic} in the sense of~\cite[\S~I.2]{Ma69}.
The latter condition guarantees, in particular, that all weight subspaces $V^\lambda \leq V$ corresponding to nonzero weights $\lambda$ have dimension $1$.
Denote by $\Lambda(\pi)$ the set of weights of $\pi$ with multiplicities (i.\,e. the zero weight is repeated $\dim(V^0)$ times).

By the very construction of $V(\mu)$, there is a basis $\{e_\lambda\}$, $e_\lambda\in V^\lambda$ parameterized by elements $\lambda\in\Lambda(\pi)$ called {\it admissible basis}.
The key property of basic representations which we use in the sequel is that there are particularly simple formulae describing the action
 of elementary root unipotents $x_\alpha(\xi)$ on the vectors of admissible basis (see~\cite[Lemme~2.3]{Ma69}),
 which, moreover, can be visualized using the technique of {\it weight diagrams} (see e.\,g.~\cite{PSV98}).
 
We denote by $v^+$ the {\it highest weight vector}, i.\,e. the basis vector $e_\mu$ corresponding to the highest weight $\mu$ of $V(\mu)$.
For $v\in V(\mu)$ we denote by $v_\lambda$ the coordinate of $v$ in the expansion of $v$ in $\{e_\lambda\}$, moreover if $\pi$ is a faithful representation we can 
identify $g\in \G(\Phi, R)$ with the matrix $(g_{\lambda, \nu})$, where $g_{\lambda, \nu} = (\pi(g)\cdot e_\nu)_{\lambda}$.
The following lemma is, in essence, a relative version of Chevalley--Matsumoto decomposition which can be easily deduced from the absolute statement (cf. \cite[Theorem~1.3]{St78}).
\begin{lemma}\label{lemma:Chevalley-Matsumoto}
Let $\pi$ be a basic fundamental representation of $\G_{sc}(\Phi, R)$ with highest weight $\mu = \varpi_s$.
Assume that $g\in \G_{sc}(\Phi, R, I)$ is such that $(g\cdot v^+)_{\mu}=1$, then $g \in \U(\Sigma_s^-, I) \cdot \G_{sc}(\Delta_s, R, I) \cdot \U(\Sigma_s^+, I)$. \end{lemma}

\section{Stability conditions}\label{sec:stability-conditions}
In this section we define the stability conditions used in the statements of our decomposition theorems in section~\ref{sec:factorizations}.
First, we recall the definition of the stable rank of an ideal introduced by L.~Vaserstein in~\cite{Va69, Va71}.
As we will be mainly concerned with applications to Chevalley groups, our exposition is limited to commutative rings.
Next, in section~\ref{sec:rel-asr} we define relative version of the absolute stable range condition introduced in~\cite{EO, MKV}.
Finally, we formulate and prove several technical lemmas about the action of certain unipotent subgroups on unimodular columns under suitable stability assumptions.

\subsection{Relative stable rank}
Recall that a row $a\in{}^n\!R$ is called \emph{$I$-unimodular} if it is congruent to $(1, 0, \ldots, 0)$ modulo $I$ and its components $a_1, a_2, \ldots, a_n$ generate $R$ as an ideal.
A column $b \in R^n$ is called $I$-unimodular if its transpose $b^t$ is an $I$-unimodular row.
We denote the set of all $I$-unimodular rows (resp. columns) by $\Umd(n, I)$ (resp. $\Ums(n, I)$).
We refer to $R$-unimodular rows and columns as simply \emph{unimodular}.
It is not hard to show that for an $I$-unimodular row $a$ there exists an $I$-unimodular column $b$ such that $ab=1$ (see~\cite[\S2]{Va69}).

An $I$-unimodular row $a=(a_1, \ldots, a_{n+1})$ is called \emph{stable} if one can choose $b_1, \ldots, b_n\in I$ such that the row $(a_1+a_{n+1}b_1, \ldots, a_n+a_{n+1}b_n)$ is also $I$-unimodular. 
We say that $I$ satisfies the condition $\SR_n(I)$ if any $I$-unimodular row of length $n+1$ is stable.
For $m \geqslant n$ the condition $\SR_n(I)$ implies $\SR_m(I)$, see~\cite[Theorem~1]{Va71}.
It is also clear that the condition $\SR_n(I)$ does not depend on the choice of the ring $R$.
By definition, the \emph{stable rank} of $I$ (denoted $\sr(I)$) is the smallest natural number $n$ such that $\SR_n(I)$ holds (we set $\sr(I)=\infty$ if no such $n$ exists).

\subsection{Relativization of the absolute stable rank}\label{sec:rel-asr}
For a row $a=(a_1, \ldots, a_n)\in{}^n\!R$ denote by $\mathfrak{J}(a)$ the intersection of all maximal ideals of $R$ containing $a_1, \ldots, a_n$.
It is easy to see that a row $a\in R^n$ is unimodular if and only if $\mathfrak{J}(a)=R$. 
Clearly, for any $g\in\GL(n,R)$ one has $\mathfrak{J}(a\cdot g)=\mathfrak{J}(a)$.
\begin{dfn}\label{dfn:j-stable}
We say that a row $a=(a_1, \ldots, a_{n+1})\in{}^{n+1}\!I$ can be \emph{$I$-shortened}, if there exist $c_1, \ldots, c_n\in I$ such that
$\mathfrak{J}(a_1, \ldots, a_{n+1})=\mathfrak{J}(a_1+c_1a_{n+1}, \ldots, a_n+c_na_{n+1}).$
\end{dfn}
\begin{dfn}\label{dfn:asr}
We say that an ideal $I$ satisfies the condition $\ASR_n(I)$ if it satisfies $\SR_n(I)$ and, moreover, any row $a\in{}^{n+1}\!I$ can be $I$-shortened.
\end{dfn}

It is easy to see that $\ASR_m(I)$ implies $\ASR_n(I)$ for any $n\geqslant m$. 
By definition, the \emph{absolute relative stable rank} $\asr(I)$ is the smallest natural $n$ such that $\ASR_n(I)$ holds (again we set $\asr(I)=\infty$ if no such $n$ exists).

A priori our definition of $\ASR_n(I)$ depends on $R$.
Below we will see that in fact there is no such dependence.
The following lemma is a relative version of~\cite[Lemma~8.2]{MKV}. 
\begin{lemma}\label{lemma:relative-asr-unimod}
For a commutative ring $R$ and an ideal $I \trianglelefteq R$ the following statements are equivalent:
\begin{lemlist} 
\item\label{asr-j-stable} Any row $a\in{}^{n+1}\!I$ can be $I$-shortened;
\item\label{asr-bak-like} For any $I$-unimodular row $(b, a_1, \ldots, a_n, d)\in\Umd(n+2, I)$ there exist $c_1, \ldots, c_n\in I$ 
 such that $(b+b', a_1+c_1d, \ldots, a_n+c_nd)$ is $I$-unimodular for any $b'\in J$, where $J=I \cdot a_1 + \ldots + I \cdot a_n + I \cdot d\leqslant I$.
\end{lemlist} 
\end{lemma}
\begin{proof}
Assume first that any row $a\in{}^{n+1}\!I$ can be $I$-shortened. 
In particular, for a given $I$-unimodular row $(b, a_1, \ldots, a_n, d)\in\Umd(n+2, I)$ there exist $c_1, \ldots, c_n$ such that
\[\mathfrak{J}(a_1, \ldots, a_{n+1})=\mathfrak{J}(a_1+c_1a_{n+1}, \ldots, a_n+c_na_{n+1}).\]
Therefore $(b,a_1+c_1d,\ldots,a_n+c_nd)$ is also unimodular. Of course, for any $b'\in J$ we could replace $b$ with $b+b'$ from the very start.

To show the converse take an arbitrary row $(a_1, \ldots, a_n, d)\in{}^{n+1}\!I$ and consider the $I$-unimodular row $(1, a_1, \ldots, a_n, d)\in\Umd(n+2, I).$
By the hypothesis, there exist $c_1, \ldots, c_n\in I$ such that
\[ v=(1+b', a'_1, \ldots, a'_n)=(1+b', a_1+c_1d, \ldots, a_n+c_nd) \]
is unimodular for any $b'\in J$.
Assume that there exists a maximal ideal $\mathfrak{m}\trianglelefteq R$ such that all $a'_1, \ldots, a'_n$ are contained in $\mathfrak{m}$, but at least one of the elements $d$, $a_i$ is not.
Then clearly $d\notin\mathfrak{m}$ and $I\not\subseteq \mathfrak{m}$ (otherwise $a_i=a'_i-c_id\in\mathfrak{m}$, contrary to the assumption).
Now we can find $t\in I$ such that its image $\bar{t}$ in the residue field $R/\mathfrak{m}$ equals $-\bar{1}/\bar{d}$.
This means that $1 + b' \in \mathfrak{m}$ for $b'=td\in J$, which contradicts the unimodularity of $v$.
This shows that no such $\mathfrak{m}$ may exist and, therefore, $\mathfrak{J}(a'_1, \ldots, a'_n)=\mathfrak{J}(a_1, \ldots, a_n, d)$.
\end{proof}

Obviously, the second statement of \cref{lemma:relative-asr-unimod} does not depend on $R$, hence, as suggested by the notation, $\asr(I)$ is independent of $R$.

Let $R$ be a commutative ring. We denote by $\Max(R)$ its \emph{maximum spectrum}, i.e. the set of maximal ideals of $R$, equipped with the Zariski topology. For a topological space $X$ denote by $\dim(X)$ its usual topological dimension.
From the definition of $\asr(I)$ and \cite[Theorem~2.3]{EO} (or~\cite[Theorem~3.7]{MKV}) it follows that
\begin{equation} \label{sr-estimates} \sr(I)\leqslant\asr(I)\leqslant\asr(R)\leqslant \dim(\Max(R))+1\leqslant\dim(\Spec(R))+1. \end{equation}

\subsection{Action of unipotent radicals}\label{sec:ur-action}
In this section we work with natural representations of classical groups (i.\,e. representations with highest weight $\varpi_1$).
It will be convenient for us to number the weights of these representations as in~\cite[\S~1B]{St78}:
For example, we write $1$ instead of $\varpi_1$, $2$ instead of $\varpi_1-\alpha_1$ etc.

Let $\lambda_1, \lambda_2 \in \Lambda(\pi)$ be a pair of weights of a representation $\pi$ such that $\lambda_1-\lambda_2\in \Phi$.
For notational convenience we write $x_{\lambda_1, \lambda_2}(\xi)$ instead of $x_{\lambda_1-\lambda_2}(\xi)$.
For example, for $\Phi=\rA_\ell$ the notation $x_{1, 2}(\xi)$ has the same meaning as $x_{\varpi_1 - \varpi_1 + \alpha_1}(\xi) = x_{\alpha_1}(\xi)$.

\begin{lemma}\label{lemma:PSV-symplectic-trick}
 Let $v=(v_1, \ldots, v_n)^t$ be a column. Denote by $v'$ the vector composed of squares of the components of $v$, i.\,e. $v'=(v_1^2, \ldots, v_n^2)^t$.
 Then for any matrix $b \in M(n, I)$ one can find a symmetric matrix $a \in M(n, I)$, $a=a^t$ such that $b \cdot v' = a \cdot v$. \end{lemma}
\begin{proof}
Straightforward computation shows that the assertion of lemma holds for the matrix $a$ defined by
\begin{equation*}
a_{ij} = b_{ij} v_{j} + b_{ji} v_{i},\ j\neq i,\quad a_{ii} = b_{ii} v_{i} - \sum\limits_{\mathclap{j=1,\ j\neq i}}^n b_{ji} v_{j}. \qedhere
\end{equation*}
\end{proof}

Let $v\in V=R^{2\ell}$ be a vector of the natural representation of $\G(\rD_\ell, R)$.
Denote by $v_+$ and $v_-$ the upper and the lower halves of $v$, i.\,e. $v_+=(v_1, \ldots, v_\ell)^t$, $v_-=(v_{-\ell}, \ldots, v_{-1})^t$.
\begin{lemma}\label{lemma:asrUnip}
Assume that one of the following holds:
\begin{lemlist}
\item \label{item:asrUnipC} $\Phi=\rC_\ell$ and $\sr(I) \leqslant \ell$;
\item \label{item:asrUnipD} $\Phi=\rD_\ell$ and $\asr(I)\leqslant \ell -1$.
\end{lemlist}
Then for any $I$-unimodular column $v=(v_+, v_-)^t\in\Ums(2\ell, I)$ there exists $g\in\U(\Sigma^+_\ell, I) \leq \E(\Phi, R, I)$ 
such that $(g \cdot v)_+ \in \Ums(\ell, I)$.
\end{lemma}
\begin{proof} 
\textsc{Case $\Phi=\rC_\ell$.}
Denote by $p$ the matrix of size $\ell$ such that its only nonzero entries equal $1$ and are located on the skew-diagonal, i.\,e. $p_{ij}=\delta_{i, \ell-j+1}$. 
For $b \in M(\ell, I)$ set $g(b)=\left(\begin{smallmatrix} e_\ell & p \cdot b \\ 0 & e_{\ell} \end{smallmatrix}\right)$.
Clearly, if $b$ is symmetric then $g(b)$ lies in $\U(\Sigma_\ell^+, I)\leq \E(\rC_\ell, R, I)$.

Notice that the column $v'=(v_1,\ldots,v_\ell, v_{-\ell}^2,\ldots,v_{-1}^2)^t$ is $I$-unimodular.
By the definition of the relative stable rank we can find a matrix $b\in \M(\ell,I)$ such that the upper half $v''_+$ of the vector $v''= g(b) \cdot v'$ is $I$-unimodular.
It is clear that $v''_+ = v_+ + pb v'_-$. 
Finally, applying \cref{lemma:PSV-symplectic-trick}, we find a symmetric matrix $a$ such that
\[ (g(a)\cdot v)_+=v_+ + pav_- = v_+ + pbv'_- = v''_+ \in \Ums(\ell, I). \]

\textsc{Case $\Phi=\rD_\ell$.} Denote by $J$ the ideal of $R$ spanned by the components of $v_{-}$. Clearly, $J \subseteq I$.
Since $\sr(I/J) \leqslant \ell-1$, the elementary group $\E(\rA_{\ell-1}, R/J, I/J)$ acts transitively on $\Ums(\ell, I/J)$ (see e.\,g.~\cite[Theorem~2.3c]{Va69}). 
This implies the existence of an element $h\in \E(\Delta_\ell, R, I)$ such that the vector $v' = h \cdot v$ satisfies $v'_i \equiv \delta_{i1} \pmod J$ for $i=1, \ldots, \ell$.

Clearly, $(v'_1, v'_{-\ell}, \ldots, v'_{-1})^t$ is $I$-unimodular.
Applying \cref{asr-bak-like}, we find $c_2, \ldots, c_\ell\in I$ such that for $v''= \prod_{i=2}^{\ell}x_{-i, -1}(c_i)\cdot v'$ one has
$(v''_1, v''_{-\ell},\ldots, v''_{-2})^t\in\Ums(\ell+1, I)$.
Now, applying the condition $\sr(I) \leqslant \ell-1$ once again, we find
$d_1, d_3, \ldots, d_{\ell}\in I$ such that the entries $(v'''_1, v'''_{-\ell}, \ldots, v'''_{-3})^t$
of $v'''=x_{1, -2}(d_1) \cdot \prod_{i=3}^{\ell} x_{-i, -2}(d_i) \cdot v''$ form an $I$-unimodular column.

We can find $f_1, f_3,\ldots, f_\ell \in R$ such that $f_1v'''_1+\sum_{i=3}^\ell f_i v'''_{-i} = 1$.
Set $\xi = v'''_1-v'''_2-1 \in I$, $v^{(4)}=x_{2, 1}(\xi f_1) \cdot \prod_{i=3}^\ell x_{2,-i}(\xi f_i) \cdot v'''$.
Clearly $v^{(4)}_2 = v^{(4)}_1-1$, therefore $v^{(4)}_+$ is $I$-unimodular.
Summarizing the above, we have found $g\in \EP_\ell(R, I)$ such that $v^{(4)}=g \cdot v$
and the assertion of the lemma immediately follows from the Levi decomposition. \end{proof}

\begin{lemma} \label{lemma:uraction} 
Let $\Phi=\rA_\ell, \rC_\ell, \rD_\ell$. Denote by $\pi$ the natural representation of $\G(\Phi, R)$ on $V=R^n$, $n=\ell+1, 2\ell, 2\ell$ respectively.
Assume that one of the following assumptions holds:
\begin{lemlist}
 \item \label{item:uractionA} $\Phi=\rA_\ell$, $\Gamma=\{ k+1, \ldots, \ell+1\} \subset \Lambda(\pi)$ and $\sr(I)\leq k\leq \ell$;
 \item \label{item:uractionC} $\Phi=\rC_\ell$, $\Gamma=\{-\ell,\ldots, -2, -1\} \subset \Lambda(\pi)$ and $\sr(I)\leq \ell$;
 \item \label{item:uractionD} $\Phi=\rD_\ell$, $\Gamma=\{-\ell,\ldots, -2, -1\} \subset \Lambda(\pi)$ and $\asr(I)\leq \ell-1$. 
\end{lemlist}
Then for any $g\in \G(\Phi, R, I)$ there exist $x\in \U(\Phi^+, I)$, $y\in \U(\Phi^-, I)$ such that $(yxg \cdot v^+)_\lambda = 0$ for all $\lambda\in \Gamma$.
\end{lemma}
\begin{proof} Denote by $v$ the first column of $g$ (i.\,e. $v=\pi(g)\cdot v^+$).

\textsc{Case $\Phi=\rA_\ell$.} From the definition of relative stable rank it follows that we can find 
$x= \left(\begin{smallmatrix} e_k & a \\ 0 & e_{n-k} \end{smallmatrix}\right) \in \U(\Sigma_k^+, I)$ such that 
the upper $k$ components of $v'= x \cdot v$ form an $I$-unimodular column. 
Now, to obtain zeroes at the desired positions it remains to subtract from $v'_{k+1},\ldots, v'_{\ell+1}$ suitable multiples of $v'_1, \ldots, v'_k$.
This operation corresponds to the left multiplication by some element $y\in\U(\Sigma_k^-, I)$.

\textsc{Case $\Phi=\rC_\ell$.} 
Applying \cref{item:asrUnipC}, we find $x \in \U(\Sigma_\ell^+, I)$ such that the upper half $v'_+$ of $v' = x \cdot v$ is unimodular.
Set $g(a) = \left(\begin{smallmatrix} e_\ell & 0 \\ p \cdot a & e_{\ell} \end{smallmatrix}\right)$.
Clearly, if $a$ is symmetric, then $g(a) \in \U(\Sigma_\ell^-, I)$.
Since the column $v''_+ = ({v'_1}^2, \ldots, {v'_\ell}^2)^t$ is $I$-unimodular, there exists a matrix $b \in M(\ell, I)$ such that $v'_- + p b v''_+ = 0$.
Finally, using \cref{lemma:PSV-symplectic-trick}, we find a symmetric matrix $a$ such that $(g(a) \cdot v')_- = p a v'_+ + v'_- = p b v''_+ + v'_- = 0$.

\textsc{Case $\Phi=\rD_\ell$.} From the proof of \cref{item:asrUnipD} it follows that there exists $h_1 \in \EP_\ell(R, I)$ such that for $v'=h_1\cdot v$ one has $v'_2=v'_1-1\in I$.
Clearly, for $v'' = z_{-\alpha_{1}}(-v'_2, -1)\cdot v'$ one has $v''_1=1$, hence by \cref{lemma:Chevalley-Matsumoto} there exists $h_2 \in \U(\Phi^-, I)$ such that the element
$g'=h_2 \cdot z_{-\alpha_{1}}(-v'_2, -1) \cdot h_1 \cdot g$ fixes $v^+$. 
Using the Levi decomposition we can write $g'=h \cdot y \cdot x \cdot g$ for some $y\in\U(\Sigma^-_\ell, I)$, $x \in \U(\Sigma^+_\ell, I)$, $h\in\E(\Delta_\ell, R, I)$.
It is clear that $x$, $y$ are the required elements.
\end{proof}

\section{Relative parabolic factorizations} \label{sec:factorizations}
In this section we formulate and prove relative versions of the decompositions from~\cite{St78}, which will be our main technical tools throughout the next section.

Let $G$ be a group and $A$ its subset. Denote by $L(A)$, $R(A)$ the left and the right stabilizers of $A$ and by $N(A)$ the normalizer of $A$:
\[ L(A) =  \left\{ g\in G\ \middle |\ g \cdot A = A \right\}\!,\ \ R(A) =  \left\{ g\in G\ \middle |\ A \cdot g = A \right\}\!,\ \ N(A) =  \left\{ g\in G\ \middle |\ g \cdot A \cdot g^{-1} = A \right\}.\]
It is easy to see that $L(A)$, $R(A)$ and $N(A)$ are subgroups of $G$. The following obvious lemma will be used several times in the sequel.
\begin{lemma} \label{lem:LN} The subgroup $L(A)$ is normalized by $N(A)$. Moreover, one has $R(A) \cap N(A) \subseteq L(A)$, $L(A) \cap N(A) \subseteq R(A)$. \end{lemma}
\begin{comment}
\begin{proof}
Let $h \in L(A)$, $g \in N(A)$. The first claim follows from the following chain of equalities:
\begin{equation} h^g \cdot A = g^{-1} \cdot h \cdot g \cdot A \cdot g^{-1} \cdot g = g^{-1} \cdot h \cdot A \cdot g = {A}^g = A. \end{equation} 
Now if $g \in R(A) \cap N(A)$, we obtain
$g \cdot A = g \cdot A \cdot g^{-1} \cdot g = A \cdot g = A,$ which implies the second claim.
\end{proof}
\end{comment}

\subsection{Relative Gauss decomposition}\label{sec:gauss}
The proof of the Gauss decomposition presented below is similar to the proof in the absolute case (cf.~\cite[Theorem~5.1]{Sm12}).

\begin{prop} \label{thm:Gauss}
Let $\Phi$ be a reduced irreducible root system of rank $\ell > 1$ and let $\Delta_1$, $\Delta_\ell$ be
the reductive subsystems of $\Phi$ corresponding to the endnodes of the Dynkin diagram of $\Phi$.
Suppose that both relative elementary subgroups $\E(\Delta_i, R, I)$, $i=1, \ell$ admit Gauss decomposition:
\[ \E(\Delta_i, R, I) = \Hh(\Delta_i, R, I) \cdot \U(\Delta^+_i, I) \cdot \U(\Delta^-_i, I) \cdot \U(\Delta^+_i, I), \quad i=1, \ell. \]
Then $\E(\Phi, R, I)$ also admits Gauss decomposition:
\begin{equation} \label{eq:gauss1} \E(\Phi, R, I) = \Hh(\Phi, R, I) \cdot \U(\Phi^+, I) \cdot \U(\Phi^-, I) \cdot \U(\Phi^+, I). \end{equation}
\end{prop}
\begin{proof}
For a closed subset of roots $S \subseteq \Phi^+$ set $A(S, I) = \U(S, I) \cdot \U(-S, I) \cdot \U(S, I)$ (here by $-S$ we denote the corresponding subset of opposite roots).
Notice that from Levi decomposition it follows that $A(\Phi^+, I) = A(\Delta_i^+, I) \cdot A(\Sigma_i^+, I)$.

Denote by $A$ the product of subgroups in the right-hand side of~\eqref{eq:gauss1}.
First of all, notice that for $h \in \Hh(\Phi, R, I)$, $\beta \in \Phi$ and $\xi \in R$ one has $x_\beta(\xi)^h = x_\beta(\xi + s\xi)$ for some $s\in I$.
From this and the assumption of the proposition we obtain for $\alpha \in \Delta_1 \cup \Delta_\ell$, $\xi \in R$ that
\begin{multline} \nonumber
 A^{x_\alpha(\xi)} = (\Hh(\Phi, R, I) \cdot A(\Phi^+, I))^{x_\alpha(\xi)} \subseteq (\Hh(\Phi, R, I) \cdot \E(\Delta_i, R, I) \cdot A(\Sigma_i^+, I))^{x_\alpha(\xi)} \subseteq \\
  \subseteq \Hh(\Phi, R, I) X_{\alpha}(I) \cdot \E(\Delta_i, R, I) \cdot A(\Sigma_i^+, I) \subseteq \Hh(\Phi, R, I) \cdot \Hh(\Delta_i, R, I) A(\Delta_i^+, I) \cdot A(\Sigma_i^+, I) \subseteq A.
\end{multline}
Thus, $A$ is normalized by root subgroups $X_\alpha(R)$, $\alpha \in \Delta_1 \cup \Delta_\ell$ and therefore is normalized by $\E(\Phi, R)$.
On the other hand, for $\beta \in \Phi^+$ we have
\begin{equation} \nonumber
 X_{\beta}(I) \cdot A \subseteq \Hh(\Phi, R, I) \cdot X_{\beta}(I) A(\Phi^+, I) \subseteq A.
\end{equation}
Thus, from $\U(\Phi^+, I) \subseteq L(A)$ and~\cref{lem:LN} we obtain that $\mathcal{X} \cdot A \subseteq A$ and hence
 that $\E(\Phi, R, I) \cdot A \subseteq A$.
\end{proof}

\begin{thm}\label{thm:srRI1}
Let $\Phi$ be a root system, let $I$ be an ideal of an arbitrary commutative ring $R$, and assume that $\sr(I)=1$.
Then the relative elementary Chevalley group $\E(\Phi, R, I)$ admits Gauss decomposition:
\[ \E(\Phi, R, I) = \Hh(\Phi, R, I) \cdot \U(\Phi^+, I) \cdot \U(\Phi^-, I) \cdot \U(\Phi^+, I). \]
\end{thm}
\begin{proof}
In view of \cref{thm:Gauss} it suffices to show that Gauss decomposition holds in the special case $\Phi=\rA_1$.
Let $g=\begin{psmallmatrix}a & b \\ c & d\end{psmallmatrix}$ be an element of $\SL(2, R, I)$.
The first column of $g$ is $I$-unimodular, therefore there exists $z\in I$ such that $a+cz\in R^*$.
Multiplying $g$ on the left by $x_{12}(z)$, we get a matrix $g'=x_{12}(z)\cdot g=\begin{psmallmatrix}a' & b' \\ c & d\end{psmallmatrix}$ with invertible element $a'$ in the top-left corner.
After multiplying $g'$ on the left by $x_{21}(-c/a')$ and on the right by $x_{12}(-b'/a')$ we get a diagonal matrix. 
Thus, we have obtained the sought Gauss decomposition for $g$:
\begin{equation*}
g=x_{12}(-z)\cdot x_{21}(c/a')\cdot
\begin{pmatrix} \varepsilon & 0 \\ 0 & 1/\varepsilon \end{pmatrix}
\cdot x_{12}(b'/a')=x_{12}(-z)\cdot
\begin{pmatrix} \varepsilon & 0 \\ 0 & 1/\varepsilon \end{pmatrix}
\cdot x_{21}(y) \cdot x_{12}(b'/a'), 
\end{equation*}
for some $\varepsilon\in 1+I$ and $y\in I$. \end{proof}

\subsection{Relative Dennis--Vaserstein decompositions}\label{sec:dennis-vaserstein}
Let $\Phi$ be an irreducible root system of rank $\ell\geq 2$ and $\{r, s\}$ be a pair of distinct endnodes of the Dynkin diagram of $\Phi$.

Before we proceed with the proof of~\cref{thm:DennisVaserstein} let us remark that in the special case $\ell=2$, $\sr(I)= 1$
 we already known that $\E(\Phi, R, I)$ admits Gauss decomposition.
It is not hard to see that Dennis--Vaserstein decomposition follows from Gauss decomposition.
Thus, in the sequel we may assume, without loss of generality, that $\ell > 2$.

From Levi decomposition~\eqref{rel:Levi-decomp} it follows that the following subsets of $\E(\Phi, R, I)$ are equal:
\begin{multline*}
\U(\Phi^+, I)\cdot \U(\Phi^-, I) \cdot \E(\Delta_r, R, I) \cdot \EP_s(R, I) = \\
= \U(\Sigma_r^+, I)\cdot \U(\Sigma^-_r, I) \cdot \E(\Delta_r, R, I) \cdot \EP_s(R, I) = \hspace{5em} \\
\hspace{5em} = \EP_r(R, I) \cdot \E(\Delta_s, R, I) \cdot \U(\Sigma_s^-, I)\cdot \U(\Sigma_s^+, I) = \\
= \EP_r(R, I) \cdot \U(\Sigma^-_r \cap \Sigma^-_s, I) \cdot \EP_s(R, I).
\end{multline*}
Denote the above subset by $A_{rs}$. 
Notice that the equality $A_{rs} = \E(\Phi, R, I)$ implies $A_{sr} = \E(\Phi, R, I)$, 
 therefore it suffices to consider only the possibilities for $\Phi$, $s$, $r$ listed in~\cref{table:dv-reps}.

\begin{table}[htb]
\[\begin{array}{l @{\qquad} l @{\qquad} c @{\quad} c @{\quad} c @{\qquad} c @{\qquad} c @{\qquad} c}
\Phi                                  & s    &r      & \mathrm{dim}(\pi) & \text{type} & \text{$\Delta_r$} & \mathrm{dim}(\pi') & |\Gamma|  \\ \hline\vphantom{\Bigl(}
\rA_\ell, \ \ell\geqslant 2           & 1    &\ell   & \ell+1         & \text{natural}       & \rA_{\ell-1}              & \ell           & 1  \\     
\rB_\ell, \ \ell\geqslant 2           & 1    &\ell   & 2\ell+1        & \text{natural}       & \rA_{\ell-1}              & \ell           & 1  \\     
\rC_\ell, \ \ell\geqslant 2           & 1    &\ell   & 2\ell          & \text{natural}       & \rA_{\ell-1}              & \ell           & 1  \\
\rD_\ell, \ \ell\geqslant 4           & 1    &\ell   & 2\ell          & \text{natural}       & \rA_{\ell-1}              & \ell           & 2  \\ 
\rD_\ell, \ \ell\geqslant 4           & \ell &\ell-1 & 2^{\ell-1}     & \text{half-spinor}   & \rA_{\ell-1}              & \ell           & \ell-2  \\
\rE_\ell, \ \ell=6, 7                 & \ell &2      & 27, 56         & \text{minimal}       & \rA_{\ell-1}              & \ell           & 3       \\ 
\rE_\ell, \ \ell=6, 7                 & \ell &1      & 27, 56         & \text{minimal}       & \rD_{\ell-1}              & 2(\ell-1)      & \ell-1 \\
\rF_4,                                & 4    & 1     & 26             & \text{minimal}       & \rC_{3}                   & 6              & 3 \end{array}\]
 \caption{List of representations used in the proof of \cref{thm:DennisVaserstein}.} \label{table:dv-reps}
\end{table}
For the proof~\cref{thm:DennisVaserstein} we need to show that the normalizer $N(A_{rs})$ is a sufficiently large subgroup.
This is accomplished in a series of lemmas below. For example, the following lemma describes root subgroups 
 which are contained in $N(A_{rs})$ for obvious reasons.
\begin{lemma}\label{lemma:dv-normal} 
For every $\alpha \in \Delta_{r, s} \cup \Phi^+ $ one has $X_\alpha(R) \subseteq N(A_{rs})$. \end{lemma}
\begin{proof}
Notice that for every $i$ the group $\EP_i(R, I)$ is normalized by $\EP_i(R)$, hence it is normalized by $X_\alpha(R)$, $\alpha \in S_i^+$.
Since $\E(\Delta_{r, s}, R)$ normalizes $\U(\Sigma_r^- \cap \Sigma_s^-, I)$, we obtain the required assertion for $\alpha \in \Delta_{r, s}$.

Now let $\alpha$ be a positive simple root. It is clear that $\alpha$ lies either in $\Delta_r$ or $\Delta_s$.
Assume, for example, the latter. Using Levi decomposition we obtain that $\U(\Sigma_r^- \cap \Sigma_s^-, I)^{X_\alpha(R)} \subseteq \U(\Sigma_s^-, I) \subseteq A_{rs}$, therefore $X_{\alpha}(R) \subseteq N(A_{rs})$.
Thus, $N(A_{rs})$ contains the subgroup $\U(\Phi^+, R)$ generated by $X_{\alpha}(R)$, $\alpha\in\Pi$, which completes the proof.
\end{proof}

\begin{lemma}\label{lemma:dv_unipotent} For any $1\leq i\leq n$ the following statements hold. 
\begin{thmlist} \item \label{item-dvu1} $\U(\Phi^+, I) = X_{\alpha_{i}}(I)\cdot \U(\Phi^+\setminus\{\alpha_{i}\}, I) = \U(\Phi^+\setminus\{\alpha_{i}\}, I)\cdot X_{\alpha_{i}}(I)$;
\item \label{item-dvu2} For any $\xi\in R$ one has $\U(\Phi^+\setminus\{\alpha_i\}, I)^{x_{-\alpha_{i}}(\xi)} \subseteq \U(\Phi^+, I)$;
\item \label{item-dvu3} $\U(\Phi^+, I)\cdot \U(\Phi^-, I) \subseteq \U(\Phi^+\setminus \{\alpha_i\}, I) \cdot \U(\Phi^-, I) \cdot X_{\alpha_{i}}(I) \cdot X_{-\alpha_{i}}(I)$.
\end{thmlist} \end{lemma}
\begin{proof} The first two assertions follow from Chevalley commutator formula, while the third one is a consequence of the first two. \end{proof}

The following lemma is the key point of the proof where stability assumptions are invoked.
\begin{lemma}\label{lemma:Stein_reduction}
Under the assumptions of \cref{thm:DennisVaserstein} one has $X_{-\alpha_r}(R) \subseteq N(A_{rs}).$
\end{lemma}
\begin{proof}
Let $\Phi$, $r$ and $s$ be as in Table~\ref{table:dv-reps} and let $\pi$ be the representation of $\G(\Phi, R)$ on the Weyl module $V(\mu)$ with highest weight $\mu=\varpi_s$.
The type and the dimension of $\pi$ are listed in Table~\ref{table:dv-reps}. %, see also~\cite{PSV98} where the weight diagrams of these representations are depicted.

Notice that $\Delta_r$ is an irreducible classical root system of type $\rA_{\ell-1}$, $\rC_{\ell-1}$ or $\rD_{\ell-1}$.
The restriction of $\pi$ to the subgroup $\G(\Delta_r, R)$ decomposes into a sum of Weyl modules for $\G(\Delta_r, R)$,
 moreover, each such summand corresponds to a connected component of the diagram obtained from the weight diagram of $\pi_R$ by removing all bonds marked $r$. 
Denote by $(\pi', V')$ the summand containing the highest weight vector $v^+$ of $\pi$,
 i.\,e. the summand corresponding to the top-left connected component of the diagram. Denote by $\Lambda(\pi') \subseteq \Lambda(\pi)$ the subset of weights belonging to this component.
In all cases under consideration, $\pi'$ is isomorphic to the natural representation of $\G(\Delta_r, R)$.

Denote by $\Gamma$ the subset of weights $\lambda \in \Lambda(\pi')$ such that $\lambda - \alpha_r \in \Lambda(\pi)$,
 in other words, $\Gamma$ consists of weights corresponding to vertices of the weight diagram which are incident to the removed bonds.
From the consideration of weight diagrams it is easy to determine the number of elements in $\Gamma$ (see~Table~\ref{table:dv-reps}).
Denote by $B$ the subset of $\E(\Delta_r, R, I)$ consisting of elements $g$ such that $(g \cdot v^+)_\lambda = 0$ for all $\lambda\in\Gamma$ and
set $A\coloneqq\U(\Phi^+, I)\cdot \U(\Phi^-, I) \cdot B \cdot \EP_s(R, I) \subseteq A_{rs}.$

We claim that $A = A_{rs}$. Indeed, let $g$ be an element of $\E(\Delta_r, R, I)$. 
Applying \cref{lemma:uraction} to the subsystem $\Delta_r$, we find $x\in\U(\Delta_r^+, I)$ and $y\in \U(\Delta_r^-, I)$ such that $yx\cdot g \in B$.
Consequently we obtain the required inclusion:
\begin{equation*} \U(\Sigma^+_r, I) \cdot \U(\Sigma^-_r, I) \cdot g = \U(\Sigma^+_r, I) x^{-1} \cdot \U(\Sigma^-_r, I)^{x^{-1}} y^{-1} (yxg) \subseteq \U(\Phi^+, I) \cdot \U(\Phi^-, I) \cdot B. \end{equation*}

Notice that by the definition of $\Gamma$ and Matsumoto lemma~\cite[Lemma~2.3]{Ma69} for any $s\in I$, $ g\in B$ one has $x_{-\alpha_r}(s) \cdot g \cdot v^+ = g \cdot v^+$, therefore
\[ X_{-\alpha_{r}}(I)^{B} \subseteq \U(\Phi^-, I) \cap \Stab(v^+) \subseteq \U(\Delta_s^-, I) \subseteq \EP_s(R, I). \]
From the above inclusion we immediately obtain that
\begin{equation*} X_{\alpha_r}(I) \cdot X_{-\alpha_r}(I) \cdot B \cdot \EP_s(R, I) \subseteq X_{\alpha_r}(I) \cdot B \cdot \EP_s(R, I) \subseteq B \cdot \U(\Sigma_r^+, I) \cdot \EP_s(R, I) = B \cdot \EP_s(R, I), \end{equation*}
which together with the third statement of \cref{lemma:dv_unipotent} implies that
\begin{equation*} \label{rel:sred}
A = \U(\Phi^+\setminus\{\alpha_r\}, I) \cdot \U(\Phi^-, I) \cdot B \cdot \EP_s(R, I).
\end{equation*}
Since $[B, X_{-\alpha_r}(R)] \subseteq \U(\Sigma_r^-, R) \cap \E(\Phi, R, I) = \U(\Sigma_r^-, I)$, we obtain the assertion of the lemma, indeed:
\begin{equation*} \label{rel:ninv} A^{X_{-\alpha_{r}}(R)} = \U(\Phi^+, I) \cdot \U(\Phi^-, I) \cdot B ^{X_{-\alpha_{r}}(R)} \cdot \EP_s(R, I) = A. \qedhere \end{equation*}
\end{proof}
We also need a separate lemma to deal with the symplectic case.
\begin{lemma}\label{lemma:DVST}
Assume that $\Phi=\rC_\ell$, $s=1$, $r=\ell$ and that $\sr(I) \leq \ell-1$. Then $X_{-\alpha_1}(R) \subseteq N(A_{\ell 1})$.
\end{lemma}
\begin{proof}
Denote by $C$ the set consisting of elements $g \in \EP_1(R, I)$ for which matrix entries $(g_{i, 2})$, $i=2, \ldots, \ell$ form an $I$-unimodular column of height $\ell-1$.
Set $A' = \EP_\ell(R, I) \cdot \U(\Sigma_1^- \cap \Sigma_\ell^-, I) \cdot C$.
Applying \cref{item:asrUnipC} we can find for every $g \in \EP_1(R, I)$ some element $x \in \U(\Sigma_\ell^+ \cap \Delta_1, I)$ such that $xg \in C$.  
The equality $A_{\ell 1} = A'$ follows from this, indeed, for $g\in \EP_1(R, I)$ one has
\begin{equation*} \EP_\ell(R, I) \cdot \U(\Sigma_1^- \cap \Sigma_\ell^-, I) \cdot g \subseteq 
 \EP_\ell(R, I)x^{-1}  \cdot \U(\Sigma_1^-, I) \cdot xg \subseteq A'. \end{equation*}

By the definition of $C$, for every $g \in C$ one can choose $y \in \U(\Sigma_1^+ \cap \Delta_\ell, I)$ such that $(y \cdot g)_{1, 2} = 0$.
Consequently, for every $g\in C$ one has
\[
 \EP_\ell(R, I) \cdot \U(\Sigma_1^- \cap \Sigma_\ell^-, I) \cdot g \subseteq \EP_\ell(R, I) y^{-1} \cdot \U(\Sigma_\ell^-\cap \Sigma_1^-, I)^{y^{-1}} \cdot y g
  \subseteq \EP_\ell(R, I) \cdot \U(\Sigma_\ell^-, I) \cdot y g.
\]
Notice that the matrix entry $(yg)_{1, 1}$ is invertible.
From the choice of $y$ it follows that for every $\xi\in R$ the element $g_1 \coloneqq (yg)^{x_{-\alpha_1}(\xi)}$
satisfies the assumptions of \cref{lemma:Chevalley-Matsumoto} and therefore can be rewritten as $g_1 = uh$ for some $u \in \U(\Sigma_1^-, I)$, $h \in \EP_1(R, I)$.
Thus, we obtain the assertion of the lemma, indeed:
\begin{multline*} {A_{\ell 1}}^{X_{-\alpha_{1}}(R)} \subseteq \EP_\ell(R, I)^{X_{-\alpha_1}(R)} \cdot \U(\Sigma_\ell^-, I)^{X_{-\alpha_1}(R)} \cdot \U(\Sigma_1^-, I) \cdot \EP_1(R, I) \subseteq \\
\subseteq \EP_\ell(R, I) \cdot \U(\Phi^-, I) \cdot \EP_1(R, I) \subseteq A_{\ell 1}. \qedhere \end{multline*}
\end{proof}

Now we are all set to prove the main result of this subsection.
\begin{proof} [Proof of \cref{thm:DennisVaserstein}]
From Lemmas~\ref{lemma:Stein_reduction} and~\ref{lemma:dv-normal} it follows that $N(A_{rs})$ contains $\EP_s(R)$.
Thus, by~\cref{lem:LN} we obtain that $\EP_s(R, I) \subseteq L(A_{rs})$.

Our next goal is to verify the inclusion $\U(\Sigma_s^-, I) \subseteq L(A_{rs})$.
Since $\widetilde{W}(\Delta_s) \subseteq \E(\Delta_s, R) \subseteq N(A_{rs})$ and 
 $X_{\alpha}(I) \subseteq L(A_{rs})$ for $\alpha \in \Sigma^-_s\cap \Delta_r$, 
we obtain from~\cref{lem:LN} and \cref{item-trans2} that $L(A_{rs})$ contains root subgroups $X_\alpha(I)$ for $\alpha$ lying in the orbit $O = W(\Delta_s) \cdot (\Sigma^-_s\cap \Delta_r)$.
Thus, we are left to consider three cases when $O$ does not coincide with $\Sigma^-_s$. 
\begin{enumerate} 
  \item \textit{Case $\Phi = \rB_\ell$, $s=1$, $r=\ell$.} 
 Only the sole short root of $\Sigma_1^-$ is not contained in $O$, denote it by $\gamma$.
 Set $\alpha = \alpha_2 + \ldots + \alpha_\ell,\ \beta = -\alpha_1 - 2\alpha_2 - \ldots -2\alpha_\ell.$
 It is clear that $\alpha \in \Phi^+$, $\beta, \beta + 2\alpha \in O$.
 Since $X_{\beta}(I), X_{\beta+2\alpha}(I) \subseteq L(A_{rs})$ and $X_{\alpha}(R) \subseteq N(A_{rs})$ we obtain the required inclusion $X_{\gamma}(I) \subseteq L(A_{rs})$ using Chevalley commutator formula:
 \begin{equation} \nonumber x_{\gamma}(\pm ab) = [x_{\beta}(a),\, x_{\alpha}(b)] \cdot x_{\beta + 2\alpha}(ab^2). \end{equation}
 \item \textit{Case $\Phi = \rF_4$, $s=4$, $r=1$.}
 For a root $\alpha\in \Sigma_4^-$ there are only $3$ possibilities for its length and the value of $m_4(-)$:
 \begin{itemize}
  \item $m_4(\alpha) = -1$, $\alpha$ is short (there are 8 such roots);
  \item $m_4(\alpha) = - 2$, $\alpha$ is long (there are 6 such roots);
  \item $m_4(\alpha) = - 2$, $\alpha$ is short (there is only one such root, denote it by $\gamma$).
 \end{itemize}
 Since $\Delta_r \cap \Sigma_s^-$ contains roots of the first two types, only $\gamma$ is not contained in $O$.
 Clearly, $\gamma = -\alpha_1 -2\alpha_2-3\alpha_3-2\alpha_4 = \alpha + \beta$ for $\alpha = \alpha_1 + \alpha_2 + \alpha_3\in \Phi^+,$ $\beta = -\alpha_{{\max}} \in O$.
 Since $\alpha \in \Phi^+$, $\beta, \beta + 2\alpha \in O$ we obtain the inclusion $X_{\gamma}(I) \subseteq L(A_{rs})$ by the same token as in the previous case.
 \item \textit{Case $\Phi = \rC_\ell$, $s=1$, $r=\ell$.} 
  In this case we need to invoke stability assumptions one more time.
  From~\cref{lemma:DVST} we obtain that $N(A_{rs})$ contains $X_{-\alpha_1}(R)$ and therefore coincides with $\E(\Phi, R)$.
  The required inclusion now follows~\cref{item-trans1}.
\end{enumerate}
Thus, we have shown that $\mathcal{X}\subseteq L(A_{rs})$. 
Recall from~\cref{prop:Stepanov-theorem} that $\E(\Phi, R, I)$ is generated by $\mathcal{Z}(\Sigma_s^-)$.
On the other hand, from $\mathcal{Z}(\Sigma_s^-) \subseteq \mathcal{X}^{\EP_s(R)} \subseteq L(A_{rs})^{N(A_{rs})} \subseteq L(A_{rs})$ we conclude that 
 $\E(\Phi, R, I) \subseteq L(A_{rs})$, which completes the proof.
\end{proof}

\subsection{Relative Bass--Kolster decompositions}\label{sec:bass-kolster}
The next theorem is a relative version of the so-called Bass--Kolster decomposition (cf.~\cite[Theorem~2.1]{St78}).
\begin{thm}\label{thm:BassKolster}
Let $\Phi$ be a classical root system of rank $\ell\geqslant2$, let $R$ be an arbitrary commutative ring and $I$ be its ideal, satisfying one of the following assumptions:
\[\begin{array}{l@{\quad}l@{\quad}l@{\quad}c}
\Phi = \rA_\ell, \ \ell\geqslant 2, & \sr(I) \leqslant \ell; \\
\Phi = \rC_\ell, \ \ell\geqslant 2, & \sr(I) \leqslant 2\ell-1; \\
\Phi = \rB_\ell, \rD_\ell, \ \ell\geqslant 3, & \asr(I) \leqslant \ell-1.
\end{array}\]
Then the principal congruence subgroup $\G(\Phi, R, I)$ admits decomposition:
\[ \G(\Phi, R, I)=  \U(\Phi^+, I) \cdot \U(\Phi^-, I) \cdot Z \cdot \U(\Sigma_1^-\setminus\{-\alpha_{\max}\}, I) \cdot \U(\Sigma_1, I) \cdot \G(\Delta_1, R, I), \]
where $Z = Z_{\alpha_{\max}}(I)=\left\{z_{-\alpha_{\max}}(r, 1)\ \middle|\ r\in I \right\}$.
\end{thm}
\begin{proof}

Let $g$ be an element of $\G(\Phi, R, I)$. Set $v=g \cdot v^+\in\Ums(n, I)$. 
Notice that in each case it suffices to find $g' \in \U(\Phi^-, I) \cdot \U(\Phi^+, I) \cdot g$ such that 
\begin{equation} \label{eq1} (g'\cdot v^+)_{1} = 1 + s \text{ and } (g'\cdot v^+)_{\varpi_{1}-\alpha_{max}} = s\ \text{for some}\ s\in I. \end{equation}
Indeed, set $g'' = z_{-\alpha_{max}}(-s, 1) \cdot g'$.
Obviously, one has $(g''\cdot v^+)_1 = 1$, $(g''\cdot v^+)_{\varpi_{1}-\alpha_{max}}=0$ and the conclusion of the theorem follows from \cref{lemma:Chevalley-Matsumoto}.

\textsc{Case $\Phi=\rA_\ell$, $n=\ell + 1$.}
Since $\sr(I)\leqslant\ell$, one can find $a_1, \ldots, a_\ell\in I$ such that $(v_1+a_1v_{\ell+1}, \ldots, v_\ell+a_\ell v_{\ell+1})^t=(v'_1, \ldots, v'_\ell)^t$ is $I$-unimodular.
Then there are $b_1, \ldots, b_\ell\in I$ such that $b_1v'_1+\ldots b_\ell v'_\ell=v'-1\in I$. 
Thus the vector \[ v'' = \prod_{i=1}^\ell x_{\ell+1, i}(b_i) \cdot \prod_{i=1}^\ell x_{i, \ell+1}(a_i) \cdot v \]
satisfies the equalities~\eqref{eq1}.

\textsc{Case $\Phi=\rC_\ell$, $n=2\ell$.}
Notice that the column $(v_1, \ldots, v_{-2}, v_{-1}^2)^t$ is also $I$-unimodular.
Applying the assumption $\sr(I)\leqslant 2\ell-1$, we find $c_1, c_2, \ldots, c_{-2} \in I \cdot v_{-1}$ such that upper $2\ell -1$ components of $v'=(v_1 + c_1 v_{-1}, \ldots, v_{-2} + c_{-2}v_{-1}, v_{-1})^t$ form an $I$-unimodular column.
By the choice of $c_i$ we can find suitable $d\in I$ such that $h_1 \cdot v = v'$ for
\[ h_1 = x_{1, -1}(c_1 + d) \cdot \prod_{i=2}^{-2} x_{i, -1}(c_i) \in \U(\Sigma_1^+, I). \]
We can find $f_1, f_2, \ldots, f_{-2} \in R$ such that $f_1v'_1+\sum_{i=2}^{-2} f_i v'_i = 1$.
Set $\xi = v''_1-v''_{-1}-1 \in I$, 
\[ h_2 = x_{-1, 1}\biggl(\xi f_1 + \sum_{i=2}^\ell v_1' \xi^2 f_i f_{-i}\biggr) \cdot \prod_{i=2}^{-2} x_{-1, i}(\xi f_i) \in \U(\Sigma_1^-, I). \]
Direct computation shows that the vector $v'' = h_2 \cdot v'$ satisfies equalities~\eqref{eq1}.

\textsc{Case $\Phi=\rD_\ell$, $n= 2\ell$.} 
By \cref{item:asrUnipD} we can find $h_1\in \U(\Sigma^+_\ell, I)$ such that the upper half $v'_+$ of $v'=h_1 \cdot v$ is $I$-unimodular.
Since $\sr(I)\leqslant \ell-1$, we can find $c_1$, $c_3, \ldots c_\ell \in I$ such that $(v''_1, v''_3, \ldots, v''_\ell) \in \Ums(\ell-1, I)$, where
\[ v''=h_2 \cdot x_{1, 2}(c_1) \cdot v', \quad h_2=\prod_{i=3}^\ell x_{i, 2}(c_i). \]
We can find $f_1, f_3, \ldots, f_\ell \in R$ such that $f_1v''_1+\sum_{i=3}^\ell f_i v''_{i} = 1$.
As before, set
\[ \xi = v''_1-v''_{-2}-1 \in I, \quad h_3 = x_{-2, 1}(\xi f_1) \cdot \prod_{i=3}^\ell x_{-2, i}(\xi f_i), \quad v'''=h_3 \cdot v''. \]
Clearly, $t_{1, 2}(c_1) \cdot h_1 \in \U(\Phi^+, I)$, $ h_3 \cdot h_2 \in \U(\Phi^-, I)$ and $v'''$ satisfies~\eqref{eq1}.

\textsc{Case $\Phi=\rB_\ell$, $n=2\ell+1$.} Subdivide $v\in \Ums(2\ell+1, I)$ as $v=(v_+, v_0, v_-)\in R^\ell\times R\times R^\ell$.
Denote by $J\leqslant I$ the ideal spanned by the components of $v_-$.
Since $\sr(I/J)\leqslant \ell$, we can find $c_1, \dots, c_\ell\in I$ such that for $v' = h \cdot v$, $h = \prod_{i=1}^\ell x_{i, 0}(c_i) \in \U(\Phi^+, I)$
one has $\bar{v'}_+=(\bar{v'_1}, \ldots, \bar{v'_\ell}) \in \Ums(\ell, I/J)$ and, therefore, $(v'_+, v'_-) \in \Ums(2\ell, I)$.
Now the proof can be finished by repeating the argument for the case $\Phi=\rD_\ell$ (applied to the subset of long roots of $\rB_\ell$).
\end{proof}

It is easy to see that the proof of the above theorem is effective and gives an estimate of the total number of elementary root unipotents involved in the decomposition.
\begin{cor}\label{cor:bass-kolster-count}
In the assumptions and notation of \cref{thm:BassKolster} every element of $\G(\Phi, R, I)$ 
can be factored into a product of one element of $\G(\Delta_1, R, I)$, one element of $Z$ and at most $4|\Sigma_1|-1$ elements of $\mathcal{X}$.
\end{cor}
\begin{proof}
The assertion can be obtained by a careful analysis of the proof of the previous theorem.
Cases $\Phi=\rA_\ell, \rC_\ell$ are immediate.
In the case $\Phi=\rD_\ell$ from the proof of \cref{thm:BassKolster} one obtains that
\begin{equation*} \G(\Phi, R, I) =  \U(\Sigma_\ell^+, I) \cdot X_{\alpha_1}(I) \cdot \U(\Sigma_2^-\cap\Delta_1, I) \cdot X_{-\alpha_{\max}}(I) \cdot Z \cdot \U(\Sigma_1^-, I) \cdot \U(\Sigma_1^+, I) \cdot \G(\Delta_1, R, I). \end{equation*}
We can present an element $g$ of $\U(\Sigma_\ell^+, I)$ as a product of $g_1 \in \U(\Sigma_{1, 2}^+ \cap \Sigma_\ell^+)$ and $g_2\in \U(\Delta_{1, 2}\cap \Sigma_\ell^+)$.
An examination of the extended Dynkin diagram of $\rD_\ell$ shows that $g_2$ either centralizes or normalizes all factors of the above decomposition (except the last one)
and therefore can be moved to the right until it is consumed by $\G(\Delta_1, R, I)$.
On the other hand, $g_1$ is a product of at most $2\ell-3$ elementary unipotents, while the width of $\U(\Sigma_1^\pm, I)$ and $\U(\Sigma_2^-\cap\Delta_1)$ with respect to the elementary unipotents does not exceed $2\ell-2$ and $2\ell-4$, respectively.
Summing up these upper bounds, we obtain
$$(2\ell-3) + 1 + (2\ell - 4) + 1 + 2\cdot (2\ell - 2) = 8\ell - 9 = 4|\Sigma_1| - 1.$$

The estimate in the case $\Phi=\rB_\ell$ can be obtained in a similar way. \end{proof}
\begin{cor}\label{cor:bass-kolster-iterated}
Assume that $\Phi$ and $I$ satisfy one of the following assumptions
\[\begin{array}{l@{\quad}l@{\quad}l}
\Phi=\rA_\ell, & \sr(I)\leqslant 2, & N'=3\left|\Phi^+\right|+2\ell - 5; \\
\Phi=\rC_\ell, & \sr(I)\leqslant 3, & N'=3\left|\Phi^+\right|+3\ell - 6; \\
\Phi=\rB_\ell, \rD_\ell, & \asr(I)\leqslant 2, & N'=4\left|\Phi^+\right| - 4.
\end{array}\]

Then every element of $\G(\Phi, R, I)$ can be decomposed into a product of one element of $\G(\langle\pm\alpha_\ell\rangle, R, I) \cong \SL(2, R, I)$ and at most $N'$ elements of $\mathcal{Z}(\Sigma_\ell)$:
\end{cor}
\begin{proof}
The assertion can be obtained by iteratively applying (for a total of $\ell-1$ times) the decomposition of \cref{thm:BassKolster} (see also Fig.~1 below).
The improved estimate for $\Phi=\rA_\ell$ (resp. $\rC_\ell$) follows from the fact that it suffices to make only two (resp. three) additions to shorten the unimodular column in the first step of the proof of \cref{thm:BassKolster}.
\end{proof}

\begin{figure}[hb]\label{fig:bass-kolster}
\[\begin{tikzcd}[row sep=tiny]
\rD_4 \arrow[out=0, in=120]{rd}{\asr\leqslant3} & & & \\
\rA_4 \arrow{r}{\sr\leqslant4} & \rA_3 \arrow{r}{\sr\leqslant3} & \rA_2 \arrow{r}{\sr\leqslant2} & \rA_1 \\
\rC_4 \arrow{r}{\sr\leqslant7} & \rC_3 \arrow{r}{\sr\leqslant5} & \rC_2 \arrow[out=0, in=-120, swap]{ru}{\sr\leqslant3} \\
\rB_4 \arrow{r}{\asr\leqslant3} & \rB_3 \arrow[out=0, in=-120, swap]{ru}{\asr\leqslant2} & &
\end{tikzcd}\] \caption{Reductions used in the proof of \cref{cor:bass-kolster-iterated} and \cref{thm:SL2width}} \end{figure}

\section{Applications} \label{sec:applications}
\subsection{Subsystem factorizations}\label{sec:subsysfact}
Recall from~\cite{LNS11, V13} that a Chevalley group over a field $\G(\Phi, F)$ can be presented as a product of at most $3|\Phi^+|$ of its subgroups $\SL(2, F)$.
As an easy corollary of Bass--Kolster decomposition we get that a similar factorization also holds for relative groups over rings of small stable rank.
\begin{thm}\label{thm:SL2width}
Let $I\trianglelefteq R$ be an ideal, $\Phi$ be an irreducible classical root system of rank $\ell$ satisfying one the following conditions:
\[\begin{array}{l@{\quad}l@{\quad}l}
\Phi=\rA_\ell, & \sr(I)\leqslant 2, & N=3\left|\Phi^+\right| - \ell - 1; \\
\Phi=\rC_\ell, & \sr(I)\leqslant 3, & N=3\left|\Phi^+\right| - 2; \\
\Phi=\rB_\ell, \rD_\ell, & \asr(I)\leqslant 2, & N=4\left|\Phi^+\right| - 3\ell.
\end{array}\]
Then the principal congruence subgroup $\G(\Phi, R, I)$ can be presented as a product of at most $N$ of its subgroups, 
 where each subgroup is an isomorphic copy of $\SL(2, R, I)$.
\end{thm}

\begin{proof}
 As in the proof of \cref{cor:bass-kolster-iterated}, the assertion is obtained via iterative application of~\cref{thm:BassKolster}.
 To reduce the number of $\SL_2$-factors involved in the decompositoin one has to group into a single $\SL_2$-factor a pair of opposite root subgroups $X_{\alpha}(I)$, $X_{-\alpha}(I)$ (or $Z_{\pm\alpha}(I)$) appearing on each
 of the $3$ junctions between the positive and negative unipotent subgroups in the Bass--Kolster decomposition.
 Since a total of $\ell-1$ reductions are used, we get the required estimate $N \leq N' - 3(\ell - 1) + 1$.
\end{proof}

We now turn our attention to the proof of \cref{thm:spin-sln-prod}, which will occupy the rest of this subsection.

Consider the decreasing chain $\Phi_k$, $k=1, \ldots, \lfloor \ell/2 \rfloor$ of root subsystems of $\Phi=\rD_\ell$ defined as follows.
If $2k \neq \ell$, let $\Phi_k$ be the subsystem of $\Phi$, spanned by the fundamental roots $\alpha_{2k-1}, \ldots, \alpha_\ell$.
Clearly, such $\Phi_k$ has type $\rD_{\ell-2k+2}$. 
In the remaining case $2k = \ell$ set $\Phi_k = \langle \alpha_\ell \rangle \cong \rA_1$.
Now let $\beta_k$ be the maximal root of $\Phi_k$, i.\,e. $\beta_k = \alpha_{\max}(\Phi_k)$, $k=1, \ldots, \lfloor \ell/2 \rfloor$.
Denote by $B$ the set of all $\beta_k$. From the definition it is clear that the elements of $B$ are mutually orthogonal to each other.
The roots $\beta_k$ can also be defined by explicit formulae:
\begin{align*}
 \beta_k =  \alpha_{2k-1} + 2\alpha_{2k}+ \ldots + 2\alpha_{\ell-2} + \alpha_{\ell-1} + \alpha_\ell, & \text{ for } k=1, \ldots, \lfloor\ell/2\rfloor-1, \\
 \beta_{\lfloor\ell/2\rfloor} = \alpha_{\ell-2}+\alpha_{\ell-1}+\alpha_\ell, & \text{ if $\ell$ is odd, } \\
 \beta_{\lfloor\ell/2\rfloor} = \alpha_\ell, & \text{ if $\ell$ is even.}
\end{align*}

\begin{lemma}\label{lemma:nikolov-weyl} There exists an element $w\in W(\rD_\ell)$ such that $w(B) \subseteq \Delta_\ell^+$. \end{lemma}
\begin{proof}
\textsc{Case $\ell=4$.} Set $w = \sigma_{\alpha_{1} + \alpha_{2}} \circ \sigma_{\alpha_{2} + \alpha_{4}}$.
Straightforward computation shows that 
\begin{align*}
& w(\beta_1) = w(\alpha_{\max}) = \sigma_{\alpha_{1} + \alpha_{2}}(\alpha_1 + \alpha_2 + \alpha_3) = \alpha_3, \\
& w(\beta_2) = w(\alpha_4) = \sigma_{\alpha_{1} + \alpha_{2}}(- \alpha_2) = \alpha_1,
\end{align*}
which implies the assertion of the lemma.

\textsc{Case $\ell \geq 5$.}
Recall from \cite[Table~9]{Dy72} that for odd (resp. even) $\ell$ all maximal subsystems of type $\rA_1+\ldots+\rA_1+\rD_3$
(resp. $\rA_1+\ldots+\rA_1+\rD_4$) are conjugate under the action of $W(\Phi)$. Consequently, we can find $w\in W(\Phi)$
such that $w(\beta_k) = \alpha_{2k-1}$ for $k < \lfloor\ell/2\rfloor$ (resp. $k < \lfloor\ell/2\rfloor-1$). 
Now using transitivity of the action of $W(\rD_3)$ on the roots (resp. by the same argument as in the case $\ell=4$) we can move the remaining root $\beta_{\lfloor\ell/2\rfloor}$ 
(resp. the remaining $2$ roots $\beta_{\lfloor\ell/2\rfloor-1}$, $\beta_{\lfloor\ell/2\rfloor}$) to
$\alpha_{\ell-1}$ (resp. to $\alpha_{\ell-3}$, $\alpha_{\ell-1}$) while fixing all the other $\beta_k$. \end{proof}

The following lemma is an analogue of Proposition~1 of~\cite{Nik07}.
\begin{lemma}\label{lemma:nikolov-type-dl}
Let $\Phi=\rD_\ell$, $\ell\geq 2$ and let $I$ be an ideal of a  commutative ring $R$.
There exist an element $y\in\E(\Phi, R)$ and an element $w\in\widetilde{W}(\Phi)$ such that
\[ \U(\Sigma_\ell^+, I)\subset[\U(\Delta_\ell^-, I), y]\cdot{}^w\!\U(\Delta_\ell^+, I). \]
\end{lemma}
\begin{proof}
Since $\U(\Sigma_\ell^+, I)$ is abelian, we can decompose it as $\U(\Sigma_\ell^+, I)=\U(\Sigma_\ell^+\setminus B, I) \cdot \U(B, I)$. 
Set $y=\prod_{\beta\in B}x_\beta(1)$. 
We will now show by induction on $\ell$ that 
\begin{equation}\label{eq:ind-stat} \U(\Sigma_\ell^+\setminus B, I)\subset[\U(\Delta_\ell^-, I), y]\cdot\U(B, I). \end{equation}
The induction base in the cases $\ell=2, 3$ is trivial.

Notice that $\beta_1$ is the only root of $\Phi$ satisfying $m_2(\beta_1)=2$, therefore Chevalley commutator formula implies
\[ \bigl[\U(\Delta_2^-, I), x_{\beta_1}(1)\bigr]=1. \]
There is no root of the form $\gamma=\alpha+\beta$ with $\alpha\in\Sigma_2^-\cap\Delta_\ell$ and $\beta\in B\setminus\{\beta_1\}$, because such a root $\gamma$ must satisfy $m_2(\gamma)=-1$ and $m_\ell(\gamma)=1$. Thus the commutator formula gives
\[\Bigl[\U(\Sigma_2^-\cap\Delta_\ell, I), \prod_{i\neq1}x_{\beta_i}(1)\Bigr]=1. \]
Since $B\setminus\{\beta_1\}\subset\Sigma_\ell^+\cap\Delta_2$, the above two identities imply
\[
\Bigl[ \U(\Sigma_2^-\cap\Delta_\ell, I)\cdot\U(\Delta_{2, \ell}^-, I), x_{\beta_1}(1)\cdot\prod_{i\neq1}x_{\beta_i}(1) \Bigr]
\equiv \bigl[ \U(\Sigma_2^-\cap\Delta_\ell, I), x_{\beta_1}(1) \bigr] \bmod \U(\Sigma_\ell^+\cap\Delta_2, I).
\]
Every element $u\in\U(\Sigma_2^-\cap\Delta_\ell, I)$ can be decomposed as $u=vw$ for some $v\in\U(\Sigma_1^-\cap\Sigma_2^-\cap\Delta_\ell, I)$ and $w\in\U(\Sigma_2^-\cap\Delta_{1, \ell}, I)$.
Using the identity
\begin{equation}\label{eq:comm-ab-c}
[ab, c]={}^a[b, c]\cdot[a, c], 
\end{equation}
we can rewrite
\[ [vw, x_{\beta_1}(1)] = {}^v[w, x_{\beta_1}(1)]\cdot[v, x_{\beta_1}(1)].  \]
Since $\U(\Sigma_1^-\cap\Sigma_2^-\cap\Delta_\ell, I)$ and $\U(\Sigma_2^-\cap\Delta_{1, \ell}, I)$ are abelian, it is easy to see that
\[ [v, x_{\beta_1}(1)]\in\U(\Sigma_2^+\cap\Sigma_\ell^+\cap\Delta_1, I), \quad [w, x_{\beta_1}(1)]\in\U((\Sigma_1^+\cap\Sigma_\ell^+)\setminus\{\beta_1\}, I). \]
Every element of $\U(\Sigma_2^+\cap\Sigma_\ell^+\cap\Delta_1, I)$ (resp. $\U(\Sigma_1^+\cap\Sigma_\ell^+\setminus\{\beta_1\}, I)$) can be expressed as such a commutator for a suitable choice of $v$ (resp. $w$).
Indeed, set $v=x_\gamma(\xi_\gamma)\cdot v'$, $\gamma=-\alpha_1-\alpha_2$, $v' \in \U(\Sigma_1^- \cap \Sigma_2^- \cap \Delta_\ell \setminus \{\gamma\}, I)$.
Using relation \eqref{eq:comm-ab-c} and the fact that $X_\gamma(I)$ commutes with $\U(\Sigma_2^+\cap\Delta_1, I)$ we get that:
\begin{multline*}
[v, x_{\beta_1}(1)] = [x_\gamma(\xi_\gamma)\cdot v', x_{\beta_1}(1)] = {}^{x_\gamma(\xi_\alpha)}[v', x_{\beta_1}(1)] \cdot [x_\gamma(\xi_\gamma), x_{\beta_1}(1)] = \\
= [v', x_{\beta_1}(1)]\cdot x_{\beta_1-\alpha_1-\alpha_2}(\xi_\gamma) = \ldots = \prod_{\mathclap{\alpha\in\Sigma_1^- \cap \Sigma_2^- \cap \Delta_\ell}} x_{\beta_1+\alpha}(\xi_\alpha).
 \end{multline*}
It remains to note that $\Sigma_2^+ \cap \Sigma_\ell^+ \cap \Delta_1 = \beta_1 + (\Sigma_1^- \cap \Sigma_2^- \cap \Delta_\ell)$. The same argument works for $[w, x_{\beta_1}(1)]$.
Direct calculation using the commutator formula shows that
\[ {}^v\!\U(\Sigma_1^+\cap\Sigma_\ell^+\setminus\{\beta_1\}, I) \equiv \U(\Sigma_1^+\cap\Sigma_\ell^+\setminus\{\beta_1\}, I) \bmod \U(\Sigma_\ell^+\cap\Delta_2, I). \]
Summing up the above arguments, we get that
\[ [\U(\Sigma_2^-, I)\cdot\U(\Delta_{2, \ell}^-, I), y] \equiv \U((\Sigma_{1, 2}^+\cap\Sigma_\ell^+)\setminus\{\beta_1\}) \bmod \U(\Sigma_\ell^+\cap\Delta_2, I), \]
hence the inclusion~\eqref{eq:ind-stat} follows from the induction hypothesis (applied to $\Delta_{1, 2} \cong \rD_{\ell-2}$). 

Finally, we have found $a\in\U(\Sigma_\ell^+\setminus B, I)$ and $b\in\U(\Delta_\ell^-, I)$
such that $$a\in[b, y]\cdot\prod_{\beta\in B}X_\beta\subset[\U(\Delta_\ell^-, I), y]\cdot\U(B, I).$$
Now the assertion of the lemma follows from~\cref{lemma:nikolov-weyl}.
\end{proof}

\begin{proof} [Proof of \cref{thm:spin-sln-prod}]
Set $L = \E(\Delta_\ell, R, I) \leq \E(\rD_\ell, R, I)$ and denote by $\sigma$ the automorphism of $\G(\rD_\ell, R)$ induced by the 
diagram automorphism of $\rD_\ell$ swapping $\alpha_\ell$ and $\alpha_{\ell-1}$.
By \cref{thm:DennisVaserstein} one has
\begin{align*}
\E(\rD_\ell, R, I) & = \EP_{\ell}(R, I)\cdot\U(\Sigma_{\ell-1}^- \cap \Sigma_{\ell}^-, I)\cdot\EP_{\ell-1}(R, I) = \\
& = L \cdot\U(\Sigma_\ell^+, I)\cdot\U(\Sigma_{\ell-1}^- \cap \Sigma_{\ell}^-, I)\cdot \big(L \cdot \U(\Sigma_\ell^+, I) \big)^\sigma\!.
\end{align*}  
Now using \cref{lemma:nikolov-type-dl}, one can find $y_1, y_2\in\G(\rD_\ell, R)$ and $w_1, w_2\in\widetilde{W}(\rD_\ell)$ such that
\begin{align*} & L \cdot \U(\Sigma_\ell^+, I) \subset L \cdot \U(\Delta_{\ell-1}^-, I) \cdot {}^{y_1}\!\U(\Delta_{\ell-1}^-, I) \cdot {}^{w_1}\!\U(\Delta_{\ell-1}^+, I), \\
& \U(\Sigma_{\ell-1}^-\cap\Sigma_\ell^-, I) \subset \U(\Delta_\ell^+, I) \cdot {}^{y_2}\!\U(\Delta_\ell^+, I) \cdot {}^{w_2}\!\U(\Delta_\ell^-, I). \end{align*} 
Thus $\E(\rD_\ell, R, I)$ is a product of at most $9$ subgroups isomorphic to $\E(\rA_{\ell-1}, R, I)$. \end{proof}

\subsection{Bounded generation}\label{sec:boundgen}
For a group $G$ denote by $W(G, X)$ \emph{the width of $G$ with respect to a generating set $X \subseteq G$}, i.\,e. the smallest natural number $N$ 
such that every element of $G$ is a product of at most $N$ elements of $X$ or their inverses.

\begin{lemma}\label{lemma:srRI1_width}
In the assumptions of \cref{thm:srRI1} the width of $\E(\Phi,R,I)$ with respect to $\mathcal{Z}(\Pi)$ does not exceed $3\left|\Phi^+\right|+2\rk\Phi-1$.
\end{lemma}
\begin{proof}
Every element $g\in\E(\Phi,R,I)$ can be written as a product $g=u_1 h v_2 u_3$ for some $h\in\Hh(\Phi,R,I)$, $u_1,u_3\in\U(\Phi,I)$, $v_2\in\U(\Phi^-,I)$. 
Write $h=\prod_{i=1}^\ell h_{\alpha_i}(\varepsilon_i)$, $\varepsilon\in1+I$. 
In view of formula~\eqref{eq:rel-tor-elementary} the element $h_{\alpha_i}(\varepsilon_i)$ can be factored as 
 $h_{\alpha_i}(\varepsilon_i) = x_{\alpha_i}(*) z_{-\alpha_i}(*,*) x_{-\alpha_i}(*)$.
Since the torus normalizes each of $X_\alpha(I)$ (see formula~\eqref{rel:h-w}), we can rewrite $g$ in the following manner, from which the estimate follows:
\[ g\in\U(\Phi,I)\cdot\prod_{i=1}^\ell\bigl(x_{\alpha_i}(*)z_{-\alpha_i}(*,*)\bigr)\cdot \U^-(\Phi,I) \U(\Phi,I). \qedhere \] \end{proof}

The following lemma is a corollary of Theorems~5.7 and 5.8 of~\cite{LSM}.
\begin{lemma}
Let $p$ be a rational prime and $c$, $d$ be a pair of coprime integers such that $p \perp d$.
Then under the assumption of the Generalized Riemann Hypothesis there exist infinitely many primes $q\equiv c\pmod{d}$ such that $p$ is a primitive root modulo $q$.
\end{lemma}

The following lemma is a relative version of \cite[Lemma~6]{VavSmSuUnitrEng} (see also~\cite{VseUnitrZ1p}):

\begin{lemma}\label{lemma:Z1p}
Set $R=\mathbb{Z}[\sfrac{1}{p}]$ and let $I$ be an ideal of $R$.
Under the assumption of the GRH the width of $\SL(2, R, I)$ with respect to the generating set
$\mathcal{Z}(\{-\alpha_1\})=X_{12}(I)\cup X_{21}(I) \cup \{z_{21}(s, \xi) \mid s\in I,\ \xi\in R\}$
does not exceed $6$.
\end{lemma}

\begin{proof}
Clearly, $I$ is a principal ideal generated by some integer $m\in\mathbb{Z}$ not divisible by $p$.
Let $g$ be an element of $\SL(2,R,I)$. Write
\[ g=\begin{pmatrix}x & y \\ z & w\end{pmatrix},\ \text{for}\ x=p^\alpha a,\ z=p^\beta bm,\ \text{where}\ a,b,\alpha,\beta\in\mathbb{Z},\ p\nmid a,b. \]

\textsc{Case 1:} $\alpha\geqslant\beta$. 
Since $p^{\alpha-\beta}a\perp bm^2$ and $p\perp bm^2$, there exist infinitely many rational primes $q$ of the form $p^{\alpha-\beta}a+bm^2k$,
such that $p$ is a primitive root modulo $q$. 
We may assume that $q$ is prime to $b$. 
Write
\begin{align*}
g_1 & = x_{12}(mk)\cdot g =
\begin{pmatrix} p^\beta q & * \\ p^\beta bm & * \end{pmatrix}. \\
\intertext{
There exists $u\geqslant 1$ such that $p^u\equiv b\pmod q$, say $p^u=b+lq$. Then
}
g_2 & = x_{21}(ml)\cdot g_1 =
\begin{pmatrix} p^\beta q & * \\ mp^{\beta+u} & * \end{pmatrix}. \\
\intertext{
Since $g_2\equiv 1\pmod m$, we can write $p^\beta q=1+cm$ for some $c$. Now set
}
g_3 & = x_{12}\left(\dfrac{-c}{p^{\beta+u}}\right)\cdot g_2 =
\begin{pmatrix} 1 & * \\ mp^{\beta+u} & * \end{pmatrix}, \\
g_4 & = x_{21}\left(-mp^{\beta+u}\right)\cdot g_3 =
\begin{pmatrix} 1 & * \\ 0 & * \end{pmatrix}, \\
g_5 & = x_{12}\left(\dfrac{c}{p^{\beta+u}}\right)\cdot g_4 =
\begin{pmatrix} 1 & * \\ 0 & * \end{pmatrix}.
\end{align*}
Notice that $g_5=z_{21}\left(-mp^{\beta+u}, c/p^{\beta+u}\right)\cdot g_2$ hence $g=x_{12} \cdot x_{21} \cdot z_{21} \cdot x_{12}$
and the length of $g$ does not exceed $4$.

\textsc{Case 2:} $\alpha<\beta$. 
Since $\mathbb{Z}[\sfrac{1}{p}]/I$ is finite, there exists $k>0$ such that $p^k\equiv 1\pmod I$.
One can choose $k>\beta-\alpha$.
Then $k+\alpha>-k+\beta$ and we obtain that
\[ h_{12} (p^k)\cdot g =
\begin{pmatrix} p^k & 0 \\ 0 & p^{-k} \end{pmatrix}
\begin{pmatrix} p^\alpha a & * \\ p^\beta bm & * \end{pmatrix}=
\begin{pmatrix} p^{k+\alpha} a & * \\ p^{-k+\beta} bm & * \end{pmatrix}. \]
We find ourselves in the situation of the previous case, therefore, we can write $g=h_{12}\cdot x_{12} \cdot x_{21} \cdot z_{21} \cdot x_{12}$.
Finally, expressing $h=x_{21}\cdot z_{21}\cdot x_{12}$ as in \eqref{eq:rel-tor-elementary}, we get that $g=x_{21} \cdot z_{21} \cdot x_{12} \cdot x_{21} \cdot z_{21} \cdot x_{12}$.
\end{proof}

For the rest of this subsection $k$ denotes a global field and $S$ is a finite set of places on $k$. 
Let $\mathcal{O}_S$ be the Dedekind ring of arithmetic type defined by $S$ and let $I$ be an ideal of $\mathcal{O}_S$.
\begin{lemma}\label{lemma:width-dedekind}
Let $\Phi$ be an irreducible classical root system of rank $\ell \geqslant 2$.
If $k$ has a real embedding, then $\G(\Phi, \mathcal{O}_S, I)$ has finite width with respect to the generating set $\mathcal{Z}(\Sigma_\ell)$.
\end{lemma}
\begin{proof}
First of all, notice that $\asr(I) \leqslant \asr(\mathcal{O}_S) \leqslant 2$. 
By \cref{cor:bass-kolster-iterated} we can write any element 
of $G=\G(\Phi, \mathcal{O}_S, I)$ as a product of a finite number of generators from $\mathcal{Z}(\Sigma_\ell)$ and one element of 
$G_0 = \G(\{\alpha_\ell, -\alpha_\ell\}, \mathcal{O}_S, I)\cong\SL(2, \mathcal{O}_S, I)$.
Consequently, to prove the statement of the lemma it suffices to express every element 
$g = \begin{psmallmatrix}1+a & b \\ c & 1+d \end{psmallmatrix} \in G_0$
as a product of a finite number of generators contained in some rank $2$ subgroup of $G$ containing $G_0$.

From $\det(g)=1$ we conclude that $a+d=bc-ad\in I^2$. 
Recall that Vaserstein's congruence subgroup is defined as
\[ G(I, I)=\left\{ \begin{pmatrix}1+a & b \\ c & 1+d\end{pmatrix}\in\SL(2, \mathcal{O}_S)\;\middle|\; a, d\in I^2, \ b, c\in I \right\}. \]
Notice that $g_1=g\cdot z_{21}(a, 1)$ is contained in $G(I, I)$
\[ \begin{pmatrix} 1+a & b \\ c & 1+d \end{pmatrix} \cdot \begin{pmatrix} 1-a & -a \\ a & 1+a \end{pmatrix} = \begin{pmatrix} 1+ba-a^2 & b-a-ba-a^2 \\ c+a+ad-ac & 1+bc-ac \end{pmatrix} \in G(I, I). \]
For any $g'=\begin{psmallmatrix}1+a & b \\ c & 1+d\end{psmallmatrix}\in G(I, I)$ the matrix $x_{21}(-c)\cdot g'\cdot x_{12}(-b)$ lies in $\SL\left(2, \mathcal{O}_S, I^2\right)$.

By \cref{lemma:Stepanov-ideal} the group $\E\left(\Phi, \mathcal{O}_S, I^2\right)$ is contained in $\E(\Phi, I)$ for any root system $\Phi\neq\rC_\ell$ of rank $\geqslant2$.
Notice that under the assumptions of the lemma a deep result of O.~Tavgen asserts that $\E(\Phi, I)$ has finite width with respect to $\mathcal{X}$, see~\cite[Theorem~3.3]{TavgenThesis}.

In remains to consider the case $\Phi=\rC_\ell$. First of all, notice that $2abc-abd\in II^{\indexbox{2}}$, indeed,
\[ \det(g_1)=a^3d-3a^2bc+a^2bd+ab^2c+a^3+a^2b+a^2d+2abc-abd+1. \]
Consequently we obtain that
\[ g_2=x_{21}(-a-c)\cdot g_1\cdot x_{12}(a-b)\equiv
\begin{pmatrix}
1+ab-a^2 & -ab-a^2 \\ ad-ac-abc & 1-ab+a^2
\end{pmatrix}\bmod II^{\indexbox{2}}. \]
Now for $g_3=g_2\cdot z_{12}\left(a^2-ab, 1\right)$ the following congruences hold:
\begin{align*}
& g_3\equiv\begin{pmatrix} 1 & -2ab \\ -abc-a^2+ab-ac+ad & 1 \end{pmatrix}\bmod II^{\indexbox{2}}, \\
& g_4=x_{12}(2ab)\cdot g_3\equiv x_{21}\left(-abc-a^2+ab-ac+ad\right)\bmod II^{\indexbox{2}}.
\end{align*}
Thus $g_4\cdot x_{21}(*)\in\SL(2, \mathcal{O}_S, II^{\indexbox{2}})$ is contained in $\E(\rC_\ell, I)$ by \cref{lemma:Stepanov-ideal} and therefore
can be expressed as a bounded product of elements $x_\alpha(*)$.
\end{proof}
The above lemmas together with \cref{cor:bass-kolster-count,thm:Gauss} imply the following result
 which is an analogue of the results of~\cite{VseUnitrZ1p, VavSmSuUnitrEng, Tavgen91} for relative groups.
\begin{thm}\label{thm:width} Let $I$ be an ideal of a commutative ring $R$.
\begin{thmlist}
\item If $R=\mathcal{O}_S$ is a Dedekind ring of arithmetic type posessing a real embedding and $\Phi$ is classical of rank $\ell\geqslant2$, then 
$W(\G(\Phi, R, I), \mathcal{Z}(\Sigma_\ell^-))$ is finite;
\item If $\sr(I) = 1$ and $\Phi$ is an arbitrary irreducible root system, then 
\[W(\E(\Phi, R, I), \mathcal{Z}(\Pi))\leqslant 3|\Phi^+|+2\rk(\Phi)-1;\]
\item If $R = \mathbb{Z}[\sfrac{1}{p}]$ for some prime number $p$, then under the assumption of the Generalized Riemann Hypothesis one has
\begin{alignat*}{2}
& W(\E(\Phi, R, I), \mathcal{Z}(\Sigma_\ell^-))\leqslant 3|\Phi^+| + 2\ell + 1 & \text{ for } \Phi=\rA_\ell, \rC_\ell, \\
& W(\E(\Phi, R, I), \mathcal{Z}(\Sigma_\ell^-))\leqslant 4|\Phi^+| + \ell + 1 & \text{ for } \Phi=\rB_\ell, \rD_\ell.
\end{alignat*}
\end{thmlist}
\end{thm}

\printbibliography

\end{document}